\pdfoutput=1
\documentclass[11pt]{article}

\usepackage[utf8]{inputenc}
\def\final{1}
\usepackage{color}
\usepackage[small,raggedright]{subfigure}
\usepackage{url}
\usepackage{graphicx}
\usepackage{amsmath}
\usepackage{amsthm}
\usepackage{bm}
\usepackage[plainpages=false]{hyperref}
\usepackage{pinlabel}

\usepackage{titlesec}
\titleformat{\section}
{\normalfont\large\bfseries\filcenter}{\thesection.}{1 ex}{}
\titleformat{\subsection}[runin]
{\normalfont\normalsize\bfseries\filcenter}{\thesubsection.}{1 ex}{}

\usepackage[charter]{mathdesign}
\usepackage[mathcal]{eucal}

\usepackage[margin=1.2 in]{geometry}

\usepackage[square,numbers,sort&compress]{natbib}

\ifnum\final=0
\newcommand{\mynote}[1]{\marginpar{\tiny\sf #1}}
\else
\newcommand{\mynote}[1]{}
\fi

\usepackage{theomac}
\newtheoremWithMacro{theorem}{Theorem}
\newtheoremWithMacro{lemma}{Lemma}[section]
\newtheoremWithMacro{corollary}{Corollary}
\newtheoremWithMacro{prop}[lemma]{Proposition}

\newcommand{\figref}[1]{Figure \ref{fig:#1}}

\newcommand{\lemref}[1]{Lemma \ref{lemma:#1}}

\newcommand{\propref}[1]{Proposition \ref{prop:#1}}

\newcommand{\theoref}[1]{Theorem \ref{theo:#1}}
\newcommand{\secref}[1]{Section \ref{sec:#1}}

\renewcommand{\eqref}[1]{(\ref{eq:#1})}

\newcommand{\lemlab}[1]{\label{lemma:#1}}
\newcommand{\proplab}[1]{\label{prop:#1}}
\newcommand{\theolab}[1]{\label{theo:#1}}
\newcommand{\seclab}[1]{\label{sec:#1}}
\newcommand{\eqlab}[1]{\label{eq:#1}}

\renewcommand{\vec}[1]{\mathbf{#1}}

\newcommand{\iprod}[2]{\left\langle {#1},{#2}\right\rangle}

\newcommand{\bgamma}{\bm{\gamma}}

\newcommand{\Euc}{\operatorname{Euc}}
\renewcommand{\O}{\operatorname{O}}

\newcommand{\Aut}{\operatorname{Aut}}

\newcommand{\Cent}{\operatorname{Cent}}

\newcommand{\eop}{\hfill$\qed$}

\newcommand{\R}{\mathbb{R}}
\newcommand{\Z}{\mathbb{Z}}

\DeclareMathOperator{\HH}{H}
\begin{document}
\title{Frameworks with forced symmetry I: \\ Reflections and rotations}
\author{Justin Malestein\thanks{Mathematisches Institut der Universität Bonn, \url{justinmalestein@gmail.com}}
\and Louis Theran\thanks{Aalto Science Institute and Department of Computer Science, Aalto University \url{louis.theran@aalto.fi}}}
\date{}
\maketitle
\begin{abstract}
\begin{normalsize}
We give a combinatorial characterization of generic
frameworks that are minimally rigid under the additional
constraint of maintaining symmetry with respect to a finite order rotation
or a reflection.  To establish these results we develop a new technique
for deriving linear representations of sparsity matroids on colored graphs
and extend the direction network method of proving rigidity characterizations
to handle reflections.
\end{normalsize}
\end{abstract}

\section{Introduction} \seclab{intro}
A \emph{$\Gamma$-framework} is a planar structure made of \emph{fixed-length bars} connected by
\emph{universal joints} with full rotational freedom.  Additionally, the bars and joints  are symmetric with
respect to the action of a point group $\Gamma$.  The allowed motions preserve the \emph{length} and
\emph{connectivity} of the bars and \emph{symmetry} with respect to the same group $\Gamma$.  This
model is very similar in spirit to that of the \emph{periodic frameworks}, introduced in \cite{BS10},
that have recently received a lot of attention, motivated mainly by applications to zeolites \cite{SWTT06,R06};
see \cite{MT10} for a discussion of the history.

When all the allowed motions are Euclidean isometries, a framework is \emph{rigid} and otherwise it is
\emph{flexible}.  In this paper, we give a \emph{combinatorial} characterization of minimally rigid, generic
$\Gamma$-frameworks when $\Gamma$ is either a finite group of rotations or a $2$ element group
generated by a reflection. Thus, $\Gamma \subset \O(2)$ and acts linearly on $\R^2$ in
the natural way. To minimize notation,
we call these \emph{cone frameworks} and \emph{reflection frameworks} respectively.  Since these $\Gamma$
are isomorphic to some $\Z/k\Z$ for an integer $k\ge 2$, we also make this identification from now on.

\subsection{The algebraic setup and combinatorial model}
Formally, a $\Gamma$-framework is given by a triple $(\tilde{G},\varphi,\tilde{\bm{\ell}})$,
where $\tilde{G}$ is a finite graph, $\varphi$ is a $\Gamma$-action on $\tilde{G}$ that is
free on the vertices and edges, and $\tilde{\bm{\ell}} = (\ell_{ij})_{ij\in E(\tilde{G})}$
is a vector of positive \emph{edge lengths} assigned to the edges of $\tilde{G}$.  A
\emph{realization} $\tilde{G}(\vec p)$ is an assignment of points $\vec p = (\vec p_i)_{i\in V(\tilde{G})}$
such that:
\begin{eqnarray}\eqlab{lengths-1}
||\vec p_j - \vec p_i||^2 = \ell_{ij}^2 & \qquad \text{for all edges $ij\in E(\tilde{G})$}, \\
\eqlab{lengths-2}
\vec p_{\varphi(\gamma)\cdot i} = \gamma \cdot\vec p_i & \qquad
\text{for all $\gamma\in \Z/k\Z$ and $i\in V(\tilde{G})$}.
\end{eqnarray}
The set of all realizations is defined to be the \emph{realization space}
$\mathcal{R}(\tilde{G},\varphi,\tilde{\bm{\ell}})$. Since we require symmetry
with respect to a fixed subgroup $\Gamma \subset \Euc(2)$ and not all rigid motions
of the points preserve the symmetry, the correct
notion of configuration space is
$\mathcal{C}(\tilde{G},\varphi,\tilde{\bm{\ell}}) =
\mathcal{R}(\tilde{G},\varphi,\tilde{\bm{\ell}})/\Cent(\Gamma)$
where $\Cent(\Gamma)$ is the centralizer of $\Gamma$ in $\Euc(2)$.
A realization is \emph{rigid} if it is isolated in the
configuration space and otherwise \emph{flexible}.  A realization is \emph{minimally rigid}
if it is rigid but ceases to be so after removing any $\Gamma$-orbit of edges from $\tilde{G}$.
This definition of rigidity corresponds to the
intuitive one given above, since a result of Milnor \cite[Lemma 3.1]{M68} implies that if a point is not
isolated in the configuration space, there is a smooth path through it. An equivalent formulation of
rigidity is that the only continuous length-preserving, symmetry-preserving motions are
``trivial''. In this case, the trivial motions are precisely $\Cent(\Gamma)$ which is one-dimensional
for both rotational and reflective symmetry.

As the combinatorial model for cone and reflection frameworks it will be more convenient
to use colored graphs.  Similarly to \cite{MT10, TFR08, BS02}, we define a \emph{colored graph}%
\footnote{Colored graphs in this sense are also called ``gain graphs'' in the literature, e.g. \cite{Z82}.}
$(G,\bgamma)$ to be a finite,
directed graph $G$, with an assignment $\bgamma = (\gamma_{ij})_{ij\in E(G)}$
of an element of a group $\Gamma$ to each edge.
A straightforward specialization of covering space theory
(see, e.g., \cite[Section 9]{MT11}) associates $(\tilde{G},\varphi)$ with a colored
graph $(G,\bgamma)$: $G$ is the quotient of $\tilde{G}$ by $\Gamma$, and the
colors encode the covering map $\tilde{G} \to G$ via a natural map $\rho : \pi_1(G,b) \to \Gamma$.
In this setting, the choice of base vertex does not matter, and indeed, we may
define $\rho : \HH_1(G, \Z)\to \Z/k\Z$ and obtain the same theory. See
\secref{rho} for the definition of $\rho$ via the colors.

\paragraph{Remark}
For our main theorems, we require that the symmetry group act freely on vertices and edges. Removing this restriction,
one can realize, in the case of rotational symmetry, a framework where the group acts with a fixed vertex
and with inverted edges. It is easy to reduce problems of rigidity for nonfree actions to free actions, and we do
so in \secref{rotnonfree}. In the case of reflection symmetry, we can similarly extend our rigidity results
to actions with inverted edges (but not fixed vertices or edges) by reducing to free actions.
See the last remark in \secref{refllamanproof}.

\subsection{Main Theorems}
We can now state the main results of this paper.
The \emph{cone-Laman} and \emph{reflection-Laman graphs}
appearing in the statement are defined in \secref{matroid}.
\begin{theorem}[\conelaman]\theolab{cone-laman}
A generic rotation framework is minimally rigid if and only if its associated colored
graph is cone-Laman.
\end{theorem}
\begin{theorem}[\reflectionlaman]\theolab{reflection-laman}
A generic reflection framework is minimally rigid if and only if its associated colored
graph is reflection-Laman.
\end{theorem}
Genericity has its standard meaning from algebraic geometry: the set of non-generic
realizations $G(\vec p)$ is a proper algebraic subset of the potential choices for
$\vec p$.  Whether a graph is cone-Laman and reflection-Laman can be checked by efficient,
combinatorial algorithms \cite{BHMT11}, making Theorems \ref{theo:cone-laman} and
\ref{theo:reflection-laman} good characterizations as well as entirely analogous to
the Maxwell-Laman Theorem \cite{L70,M64}, which gives the combinatorial characterization for
generic bar-joint frameworks in the plane.

\paragraph*{Remark}
The approach to genericity is the one from \cite{ST10}. We choose
it because it is more computationally
effective than ones based on algebraic independence, and
also less technical.
Since genericity is a Zariski-open condition, standard results imply that:
the non-generic set of $\vec p$ has measure zero; for a generic $\vec p$,
there is an open neighborhood $U\ni \vec p$ that contains only generic points;
every generic $\vec p$ is a regular point of the rigidity map that sends $\vec p$
to the edge lengths in $G(\vec p)$ (or, indeed, any other polynomial map);
the rank of the linear system for infinitesimal motions (described below)
is maximal over all choices of $\vec p$.
Some of these are sometimes taken to be defining properties of the generic set in
the literature. The relevant algebraic geometry facts are
collected in \cite[Appendix A]{KTTU12}.

\subsection{Infinitesimal rigidity and direction networks}
As in all known proofs of ``Maxwell-Laman-type'' theorems such as Theorems
\ref{theo:cone-laman} and \ref{theo:reflection-laman},
we give a combinatorial characterization of a linearization of the problem known as
\emph{generic infinitesimal rigidity}.  Given a realization $\tilde{G}(\vec p)$, a
set of vectors $\vec v = \left(\vec v_i\right)_{i\in V(\tilde{G})}$ is an
\emph{infinitesimal motion} of $\tilde{G}(\vec p)$ if it satisfies
\begin{eqnarray}
\eqlab{infmot-1}
\iprod{\vec p_j - \vec p_i}{\vec v_j - \vec v_i} = 0 & \qquad \text{for all edges $ij\in E(\tilde{G})$} \\
\eqlab{infmot-2}
\vec v_{\varphi(\gamma) \cdot i} = \gamma \cdot \vec v_i & \qquad \text{for all $\gamma\in \Z/k\Z$}
\end{eqnarray}
The system \eqref{infmot-1}--\eqref{infmot-2} arises by computing the formal differential of
\eqref{lengths-1}--\eqref{lengths-2}.  Geometrically, infinitesimal motions preserve the edge
lengths to first order and preserve the symmetry of $\tilde{G}(\vec p)$. A cone or reflection
framework is \emph{infinitesimally rigid} if the space of infinitesimal motions is
precisely that induced by the trivial motions, i.e. $\Cent(\Gamma)$.
Equivalently,the framework is infinitesimally rigid if and only if the space of infinitesimal motions is
$1$-dimensional. As discussed  above, the rank of the system \eqref{infmot-1}--\eqref{infmot-2}, which
defines a rigidity matrix, is maximal at any generic point, so infinitesimal rigidity and rigidity
coincide generically.
In \secref{genericity2} we give a description of the non-generic set for cone frameworks; the case of
reflection frameworks is analogous.

To characterize infinitesimal rigidity, we use a
\emph{direction network} method (cf. \cite{W88,ST10,MT10}).  A
\emph{$\Gamma$-direction network} $(\tilde{G},\varphi,\vec d)$ is a
graph $\tilde{G}$ equipped with a free action $\varphi: \Gamma \to \Aut(G)$ and a symmetric assignment of directions with one direction $\vec d_{ij}$ per edge. The
\emph{realization space} of a direction network is the set of solutions $\tilde{G}(\vec p)$ to the system of
equations:
\begin{eqnarray}
\eqlab{dn-realization1}
\iprod{\vec p_j - \vec p_i}{\vec d^{\perp}_{ij}} = 0 & \qquad \text{for all edges $ij\in E(\tilde{G})$} \\
\eqlab{dn-realization2}
\vec p_{\varphi(\gamma)\cdot i} = \gamma \cdot\vec p_i & \qquad
\text{for all $\gamma\in \Z/k\Z$ and $i\in V(\tilde{G})$}
\end{eqnarray}
where, as above, $\Gamma$ is a finite group of rotations about the origin or the $2$-element group generated by
a reflection. We call such a direction network a \emph{cone direction network}
in the former case and a \emph{reflection direction network} in the latter case.
By \eqref{dn-realization2}, a  $\Gamma$-direction network is determined by assigning a direction to each
edge of the colored quotient graph $(G,\bgamma)$ of $(\tilde{G},\varphi)$ (cf. \cite[Lemma 17.2]{MT11}).
A realization of a $\Gamma$-direction network is \emph{faithful} if none of the edges of its graph
have coincident endpoints and \emph{collapsed} if all the endpoints coincide.

A basic fact in the theory of finite planar frameworks \cite{W88,ST10,DMR07}
is that, if a direction network has faithful realizations, the dimension of the realization
space is equal to that of the space of infinitesimal motions of a generic framework with the same
underlying graph.  This is also true for $\Gamma$-frameworks, if the symmetry group
contains only orientation-preserving elements, as in \cite{MT10}, or the finite order rotations
considered here. Thus, a characterization of generic cone direction networks with a $1$-dimensional space of
faithful realizations implies a characterization of generic minimal rigidity by a straightforward
sequence of steps.  We will show:
\begin{theorem}[\conedirectionnetwork]\theolab{cone-direction-network}
A generic cone direction network has a
faithful realization that is unique up to scaling
if and only if its associated colored graph is cone-Laman.
\end{theorem}
\noindent From this, \theoref{cone-laman} follows using a slight modification of the arguments in
\cite[Section 17-18]{MT10}, which we present quickly in \secref{cone-laman-proof} where we
will highlight where the proof breaks down for the reflection case. The main novelty in the
proof of \theoref{cone-direction-network} is that we make a direct geometric argument for the
key \propref{cone-collapse}, since standard results on linear representations of matroid unions
don't apply to the system \eqref{dn-realization1}--\eqref{dn-realization2}.

The situation for reflection direction networks is more complicated.
The reasoning used to transition from direction networks to
infinitesimal rigidity in the orientation-preserving case does \emph{not} apply verbatim in the
presence of reflections.  Thus, we will need to rely on a more technical analogue of
\theoref{cone-direction-network}, which we state after giving an important definition.

Let $(\tilde{G},\varphi,\vec d)$ be a direction network and define $(\tilde{G},\varphi,\vec d^{\perp})$
to be the direction network with $(\vec d^\perp)_{ij} = (\vec d_{ij}^\perp)$.  These two direction
networks form a \emph{special pair} if:
\begin{itemize}
\item $(\tilde{G},\varphi,\vec d)$ has a, unique up to scale and (vertical) translation, faithful realization.
\item $(\tilde{G},\varphi,\vec d^\perp)$ has only collapsed realizations.
\end{itemize}
\begin{theorem}[\linkeddirectionnetworks]\theolab{reflection-direction-network}
Let $(G,\bgamma)$ be a colored graph with $n$ vertices, $2n-1$ edges, and lift $(\tilde{G},\varphi)$.
Then there are directions $\vec d$ such that the direction networks $(\tilde{G},\varphi,\vec d)$
and $(\tilde{G},\varphi,\vec d^\perp)$ are a special pair if and only if $(G,\bgamma)$ is reflection-Laman.
\end{theorem}
Briefly, we will use \theoref{reflection-direction-network} as follows: the faithful realization of
$(\tilde{G},\varphi,\vec d)$ gives a symmetric immersion of the graph $\tilde{G}$ that can be interpreted
as a framework, and the fact that $(\tilde{G},\varphi,\vec d^\perp)$ has only collapsed realizations
will imply that the only symmetric infinitesimal motions of this framework correspond to translation
parallel to the reflection axis.

\subsection{Related work}
This paper is part of a sequence extending our results about periodic frameworks \cite{MT10},
and the results (and proofs) reported here have been previously circulated in other preprints:
Theorems \ref{theo:cone-laman} and \ref{theo:cone-direction-network}
in \cite{MT11}.  Theorems \ref{theo:reflection-laman} and
\ref{theo:reflection-direction-network} in \cite{MT12a}.

This paper deals with the setting of ``forced symmetry'', in which
all the motions considered preserve the $\Gamma$-symmetry of the framework.
Some directly related results are those of Jordán, Kaszanitzky, and Tanigawa \cite{JKT12},
who substantially generalize the results here to dihedral groups of order $2(2k + 1)$
using mainly inductive constructions. Tanigawa \cite{T12} proves results for
scene analysis and body-bar frameworks symmetric with respect to a large number of groups
using a more matroidal method.

Much of the interest in forced symmetry arises from the study of periodic
frameworks, which are symmetric with respect to a translation lattice.  These
appear in \cite{W88}, though recent interest should be traced to the
more general setup in \cite{BS10}.  A Maxwell-Laman-type theorem for periodic
frameworks in dimension $2$ appears in our paper \cite{MT10}. Periodic
frameworks admit a number of natural variants; in dimension $2$ there
are combinatorial rigidity characterizations for the fixed-lattice \cite{R12}
partially-fixed lattice with $1$-degree of freedom \cite{NR12} and fixed-area unit cell
\cite{MT14}. This paper emerged from a project to extend the rigidity characterizations
from \cite{MT10} to more symmetry groups; the sequel \cite{MT2014} has results for
the orientation-preserving plane groups.

In dimensions $d\ge 3$, combinatorial rigidity characterizations for even
finite bar-joint frameworks are not known.  The forced-symmetric case does not
appear to be much easier.  However, Maxwell-type counting heuristics have
been determined for a large number of space and point groups by Ross, Schulze and
Whiteley \cite{RSW11}.

All the combinatorial work mentioned so far approaches forced-symmetric frameworks
in a similar formalism.  Another approach, used in \cite{BS11,BST15} is to study
what, in our terminology, is the underlying graph of the colored graph. This leads
to a theory of a somewhat different flavor.

Another direction in the study of symmetric frameworks is to not force
the symmetry constraint.  This is the approach taken by Fowler and Guest \cite{FG99},
and a number of combinatorial characterizations are known \cite{Sch10,S10,ST14,ST13}.

\subsection{Notations and terminology}  In this paper, all graphs $G=(V,E)$ may be multi-graphs.  Typically,
the number of vertices, edges, and connected components are denoted by $n$, $m$, and $c$, respectively.
The notation for a colored graph is $(G,\bgamma)$, and a symmetric graph with a free $\Z/k\Z$-action
is denoted by $(\tilde{G},\varphi)$.
If  $(\tilde{G},\varphi)$ is the lift of $(G,\bgamma)$, we
denote the fiber over a vertex $i\in V(G)$ by $\tilde{i}_\gamma$, with $\gamma$ ranging over $\Z/k\Z$.
The fiber over a directed edge $ij \in E(G)$ with color $\gamma_{ij}$ consists of the edges
$\tilde{i}_\gamma \tilde{j}_{\gamma+\gamma_{ij}}$ for $\gamma$ ranging over $\Z/k\Z$.

We also use \emph{$(k,\ell)$-sparse graphs} \cite{LS08} and their generalizations.  For a graph $G$, a
\emph{$(k,\ell)$-basis} is a maximal $(k,\ell)$-sparse subgraph; a \emph{$(k,\ell)$-circuit} is an edge-wise
minimal subgraph that is not $(k,\ell)$-sparse; and a \emph{$(k,\ell)$-component} is a maximal subgraph
that has a spanning $(k,\ell)$-graph. (To simplify terminology, we follow a convention of
\cite{ST10} and refer to $(k, \ell)$-tight graphs simply as $(k, \ell)$-graphs.)

Points in $\R^2$ are denoted by $\vec p_i = (x_i,y_i)$, indexed sets of points by $\vec p = (\vec p_i)$,
and direction vectors by $\vec d$ and $\vec v$. For any vector $\vec v$, we denote its counter-clockwise
rotation by $\pi/2$ by $\vec v^\perp$. Realizations
of a cone or reflection direction network $(\tilde{G},\varphi,\vec d)$ are written as $\tilde{G}(\vec p)$, as are realizations
of abstract  $\Gamma$-frameworks.  Context will always make clear the type of realization under consideration.

\subsection{Acknowledgements}
LT is supported by the European Research Council under the European Union's Seventh Framework
Programme (FP7/2007-2013) / ERC grant agreement no 247029-SDModels. JM is supported by
the European Research Council under the European Union's Seventh Framework Programme
(FP7/2007-2013) / ERC grant agreement no 226135.
LT and JM have been supported by  NSF CDI-I grant DMR 0835586 to
Igor Rivin and M. M. J. Treacy.

\section{Cone- and Reflection-Laman graphs} \seclab{matroid}
In this section we define cone-Laman and reflection-Laman graphs and
develop the properties we need.  We start by recalling some general facts about
colored graphs and the associated map $\rho$.

\subsection{The map $\rho$ and equivalent colorings}  \seclab{rho}
Let $(G,\bgamma)$ be a $\Z/k\Z$-colored graph.  Since all the colored graphs in this paper have
$\Z/k\Z$ colors, from now on we make this assumption and write simply ``colored graph''.

The map $\rho : \HH_1(G, \Z)\to \Z/k\Z$ is defined on oriented cycles by adding up
the colors on the edges; if the cycle
traverses the edge in reverse, then the edge's color is added with a minus sign, and
otherwise it is added without a sign. As the notation suggests,
$\rho$ extends to a homomorphism from $\HH_1(G, \Z)$ to $\Z/k\Z$, and it is well-defined even if $G$ is not connected
\cite[Section 2]{MT10}.  The $\rho$-image of a colored graph is defined to be \emph{trivial} if it contains
only the identity.

We say $\bgamma$ and $\bm{\eta}$ are {\it equivalent colorings} of a graph $G$ if the corresponding lifts are isomorphic covers of $G$.
Equivalently, $\bgamma$ and $\bm{\eta}$ are equivalent if the induced representations $\HH_1(G, \Z) \to \Z/k\Z$ are identical.
Since all of the characteristics of rigidity or direction networks for a colored graph can be defined on the lift,
$(G, \bgamma)$ and $(G, \bm{\eta})$ have the same generic rigidity properties and the same generic rank for direction networks.
We record the following lemma which will simplify the exposition of some later proofs.

\begin{lemma} \lemlab{uncolored-subgraph}
Suppose $G'$ is a subgraph of $(G, \bgamma)$ such that $\rho(\HH_1(G', \Z))$ is trivial. Then, there is an
equivalent coloring $\bm{\eta}$ such that $\eta_{ij} = 0$ for all edges $ij$ in $G'$.
\end{lemma}
\begin{proof}
It suffices to construct colors $\eta_{ij}$ such that the induced representation $\HH_1(G, \Z) \to \Z/k\Z$ is equal
to $\rho$. There is no loss of generality in assuming that $G$ is connected.
Choose a spanning forest $T'$ of $G'$. Extend $T'$ to a spanning tree $T$ of $G$, and set $\eta_{ij} = 0$ for all $ij \in E(T)$.
Each edge $ij$ in $G - T$ creates a fundamental (oriented) cycle $C_{ij}$ with respect to $T$.
Define $\eta_{ij} = \pm \rho(C_{ij})$ with the sign being positive if the orientation of $ij$ and $C_{ij}$ agree and
negative otherwise. Since, by hypothesis, $\rho(C) = 0$ for all cycles in $G'$, all its edges are colored $0$
by $\bm{\eta}$.
\end{proof}

\subsection{Cone- and Reflection-Laman graphs}
Let $(G,\bgamma)$ be a colored graph with $n$ vertices and $m$ edges.  We define $(G,\bgamma)$ to be a
\emph{cone-Laman graph} or \emph{reflection-Laman graph}
if: $m=2n-1$, and for all subgraphs $G'$, spanning $n'$
vertices, $m'$ edges, $c'$ connected components with non-trivial $\rho$-image and $c'_0$ connected components
with trivial $\rho$-image
\begin{equation}\eqlab{cone-laman}
m'\le 2n' - c' - 3c'_0
\end{equation}
The underlying graph $G$ of a cone-Laman graph is easily seen to be a $(2,1)$-graph.
In the setting of reflection symmetry the only colored graphs arising are $\Z/2\Z$-colored; for such colorings,
the families of cone-Laman and reflection-Laman graphs, which are defined purely combinatorially, are
precisely the same. See \figref{examples} for some examples of various kinds of graphs described
here and below.

We call $(G, \bgamma)$ \emph{cone-laman-sparse}
(resp. \emph{reflection-laman-sparse}) if it satisfies \eqref{cone-laman} for all subgraphs.
Note that while \theoref{cone-laman} and \theoref{reflection-laman} imply these classes of graphs
are matroids, at this point we have not established this. Hence, we define $(G, \bgamma)$ to be a
\emph{cone-laman-circuit} (resp. \emph{reflection-laman-circuit}) if it is cone-laman-sparse
(resp. reflection-laman-sparse) after the removal of any edge.

\subsection{Ross graphs and circuits}
Another family we need is that of \emph{Ross graphs}\footnote{
This terminology is from \cite{BHMT11}.  Elissa Ross introduced this class in
\cite{R12}, and we introduced this terminology in light of her contribution.
In \cite{R12}, they are called ``constructive periodic orbit graphs''.
}.  These are colored graphs with $n$ vertices,
$m = 2n - 2$ edges, satisfying the sparsity counts
\begin{equation}\eqlab{ross}
m'\le 2n' - 2c' - 3c'_0
\end{equation}
using the same notations as in \eqref{cone-laman}.  In particular, Ross graphs $(G,\bgamma)$
have as their underlying graph, a $(2,2)$-graph $G$, and are thus connected \cite{LS08}.

A \emph{Ross-circuit}%
\footnote{The matroid of Ross graphs has more circuits, but these are the ones
we are interested in here.  See \secref{reflection-22}.} %
is a colored graph that becomes a Ross graph after removing \emph{any} edge.  The underlying
graph $G$ of a Ross-circuit $(G,\bgamma)$ is a $(2,2)$-circuit, and these are also known
to be connected \cite{LS08}, so, in particular, a Ross-circuit has $c'_0=0$ and $c' =1$, and
thus satisfies \eqref{cone-laman} on the whole graph.  Since \eqref{ross}
implies \eqref{cone-laman} and \eqref{ross} holds on proper subgraphs, we see that every Ross-circuit is reflection-Laman.
Because reflection-Laman graphs are $(2,1)$-graphs and subgraphs that are $(2,2)$-sparse
satisfy \eqref{ross}, we get the following structural result.
\begin{prop}[\xyzzy][{\cite[Lemma 11]{BHMT11}}]\proplab{ross-circuit-decomp}
Let $(G,\bgamma)$ be a reflection-Laman graph.  Then each $(2,2)$-component of $G$ contains at most one
Ross-circuit, and in particular, the Ross-circuits in $(G,\bgamma)$ are vertex disjoint.
\end{prop}

\subsection{Cone-$(2,2)$ graphs}\seclab{cone-22}
A colored graph $(G,\bgamma)$ is defined to be a \emph{cone-$(2,2)$} graph,
if it has $n$ vertices, $m=2n$ edges, and satisfies the
sparsity counts
\begin{equation} \eqlab{cone22a}
m' \le 2n' - 2c'_0
\end{equation}
using the same notations as in \eqref{cone-laman}.
The link with cone-Laman graphs is the following straightforward proposition:
\begin{prop}\proplab{cone-laman-doubling}
A colored graph $(G,\bgamma)$ is cone-Laman if and only if, after adding a copy of any
colored edge, the resulting colored graph is a cone-$(2,2)$ graph. \eop
\end{prop}

\subsection{Reflection-$(2,2)$ graphs}\seclab{reflection-22}
A colored graph $(G,\bgamma)$ is defined to be
a \emph{reflection-$(2,2)$} graph, if it has $n$ vertices, $m=2n-1$ edges, and satisfies the
sparsity counts
\begin{equation} \eqlab{ref22a}
m' \le 2n' - c' - 2c'_0
\end{equation}
using the same notations as in \eqref{cone-laman}.
The relationship between Ross graphs and reflection-$(2,2)$ graphs we will need is:
\begin{prop} \proplab{ross-adding}
Let $(G,\bgamma)$ be a Ross-graph.  Then for either
\begin{itemize}
\item an edge $ij$ with any color where $i \neq j$
\item or a self-loop $\ell$ at any vertex $i$ colored by $1$
\end{itemize}
the graph $(G+ij,\bgamma)$ or $(G+\ell,\bgamma)$ is reflection-$(2,2)$.
\end{prop}
\begin{proof}
Adding $ij$ with any color to a Ross $(G,\bgamma)$ creates either a Ross-circuit, for which
$c'_0=0$ or a Laman-circuit with trivial $\rho$-image.  Both of these types of
graph meet this count, and so the whole of $(G+ij,\bgamma)$ does as well.
\end{proof}
It is easy to see that every reflection-Laman graph is a reflection-$(2,2)$ graph.  The
converse is not true.
\begin{lemma}\lemlab{reflection-laman-vs-reflection-22}
A colored graph $(G,\bgamma)$ is a reflection-Laman graph if and only if it is a
reflection-$(2,2)$ graph and no subgraph with trivial $\rho$-image is a $(2,2)$-block.\eop
\end{lemma}
Let $(G,\bgamma)$ be a reflection-Laman graph, and let $G_1,G_2,\ldots,G_t$ be the Ross-circuits
in $(G,\bgamma)$.  Define the \emph{reduced graph} $(G^*,\bgamma)$ of $(G,\bgamma)$ to be the
colored graph obtained by contracting each $G_i$, which is not already a single vertex with a self-loop (this is necessarily
colored $1$), into a new vertex $v_i$, removing any self-loops
created in the process, and then adding a new self-loop with color $1$ to each of the $v_i$.
By \propref{ross-circuit-decomp} the reduced graph is well-defined.
\begin{prop}\proplab{reduced-graph}
Let $(G,\bgamma)$ be a reflection-Laman graph.  Then its reduced graph is a reflection-$(2,2)$ graph.
\end{prop}
\begin{proof}
Let $(G,\bgamma)$ be a reflection-Laman graph with $t$ Ross-circuits with vertex sets $V_1,\ldots,V_t$.
By \propref{ross-circuit-decomp}, the $V_i$ are all
disjoint. Now select a Ross-basis $(G',\bgamma)$ of $(G,\bgamma)$.  The graph $G'$
is also a $(2,2)$-basis of $G$, with $2n-1 - t$ edges, and each of the $V_i$ spans a $(2,2)$-block
in $G'$. The $(k,\ell)$-sparse graph
Structure Theorem \cite[Theorem 5]{LS08} implies that contracting each of the $V_i$ into a new vertex $v_i$
and discarding any self-loops created, yields a $(2,2)$-sparse graph $G^+$ on $n^+$ vertices and $2n^+ - 1 - t$
edges.  It is then easy to check that adding a self-loop colored $1$ at each of the $v_i$ produces
a colored graph satisfying the reflection-$(2,2)$ counts \eqref{ref22a} with exactly
$2n^+ -1$ edges.  Since this is the reduced graph, we are done.
\end{proof}

\subsection{Decomposition characterizations}
A \emph{map-graph} is a graph with exactly one cycle per connected component.  A \emph{cone-$(1,1)$} or \emph{reflection-$(1,1)$} graph
is defined to be a colored graph $(G,\bgamma)$ where $G$, taken as an undirected graph, is a map-graph and
the $\rho$-image of each connected component is non-trivial. Note that by \cite[Matroid Theorem]{Z82},
cone-$(1,1)$ and reflection-$(1,1)$ graphs each are bases of a matroid.
\begin{lemma}\lemlab{reflection-22-decomp}
Let $(G,\bgamma)$ be a colored graph.  Then $(G,\bgamma)$ is a reflection-$(2,2)$ graph if and only
if it is the union of a spanning tree and a reflection-$(1,1)$ graph.
\end{lemma}
\begin{proof}
By \cite[Matroid Theorem]{Z82}, reflection-$(1,1)$ graphs are equivalent to graphs satisfying
$m = n$ and
\begin{equation} \eqlab{ref11a}
m' \le n' - c'_0
\end{equation}
for every subgraph $G'$. Moreover, the right hand side of \eqref{ref11a} is
the rank function of the matroid. We can rewrite \eqref{ref22a} as
\begin{equation}\eqlab{ref22redux}
m' \le (n' - c'_0) + (n' - c' - c'_0).
\end{equation}
The second term in \eqref{ref22redux} is well-known to be the rank function of the graphic matroid, and
the lemma follows from the Edmonds-Rota construction \cite{ER66} and the Matroid Union Theorem.
\end{proof}
Nearly the same proof yields the analogous statement for cone-$(2,2)$ graphs.
\begin{prop}\proplab{cone-22-decomp}
A colored graph $(G,\bgamma)$ is cone-$(2,2)$ if and only if it is the the union of two
cone-$(1,1)$ graphs.\eop
\end{prop}

In the sequel, it will be convenient to use this slight refinement of \lemref{reflection-22-decomp}.
\begin{prop}\proplab{reflection-22-nice-decomp}
Let $(G,\bgamma)$ be a reflection-$(2,2)$ graph.  Then there is an equivalent coloring $\bgamma'$ of the edges of
$G$ such that the tree in the decomposition as in \lemref{reflection-22-decomp} has all edges colored by the identity.
\end{prop}
\begin{proof}
Apply \lemref{uncolored-subgraph} to the tree in the decomposition.
\end{proof}
\propref{reflection-22-nice-decomp} has the following re-interpretation in terms of the symmetric
lift $(\tilde{G},\varphi)$:
\begin{prop}\proplab{reflection-laman-decomp-lift}
Let $(G,\bgamma)$ be a reflection-$(2,2)$ graph.  Then for a decomposition, as provided
by \propref{reflection-22-nice-decomp}, into a spanning tree $T$ and a reflection-$(1,1)$
graph $X$:
\begin{itemize}
\item Every edge $ij\in T$ lifts to the two edges $\tilde{i}_0 \tilde{j}_0$ and $\tilde{i}_1 \tilde{j}_1$.
In particular, the vertex representatives in the lift all lie in a single connected
component of the lift of $T$.
\item Each connected component of $X$ lifts to a connected graph.
\end{itemize}
\end{prop}

\subsection{The overlap graph of a cone-$(2,2)$ graph}
Let $(G,\bgamma)$ be cone-$(2,2)$ and fix a decomposition of it
into two cone-$(1,1)$ graphs $X$ and $Y$.
Let  $X_i$ and $Y_i$ be the connected components of $X$ and $Y$, respectively.
Also select a base vertex $x_i$ and $y_i$ for each connected component of $X$ and $Y$,
with all base vertices on the cycle of their component. Denote the collection of base vertices by $B$.
We define the \emph{overlap graph} of $(G,X,Y,B)$ to be the directed graph with:
\begin{itemize}
\item Vertex set $B$.
\item A directed edge from $x_i$ to $y_j$ if $y_j$ is a vertex in $X_i$.
\item A directed edge from $y_j$ to $x_i$ if $x_i$ is a vertex in $Y_j$
\end{itemize}

The property of the overlap graph we need is:
\begin{prop}\proplab{overlap-cycle}
Let $(G,\bgamma)$ be a cone-$(2,2)$ graph. Fix a decomposition into cone-$(1,1)$
graphs $X$ and $Y$ and a choice of base vertices.
Then the overlap graph of $(G, X,Y, B)$ has a directed cycle in each connected component.
\end{prop}
\begin{proof}
Every vertex has exactly one incoming edge, since each vertex is in exactly one connected
component of each of $X$ and $Y$.  Directed graphs with an in-degree exactly one
have exactly one directed cycle per connected component (see, e.g., \cite{ST09}).
\end{proof}

\begin{figure}[htbp]
\centering
\includegraphics[width=.6\textwidth]{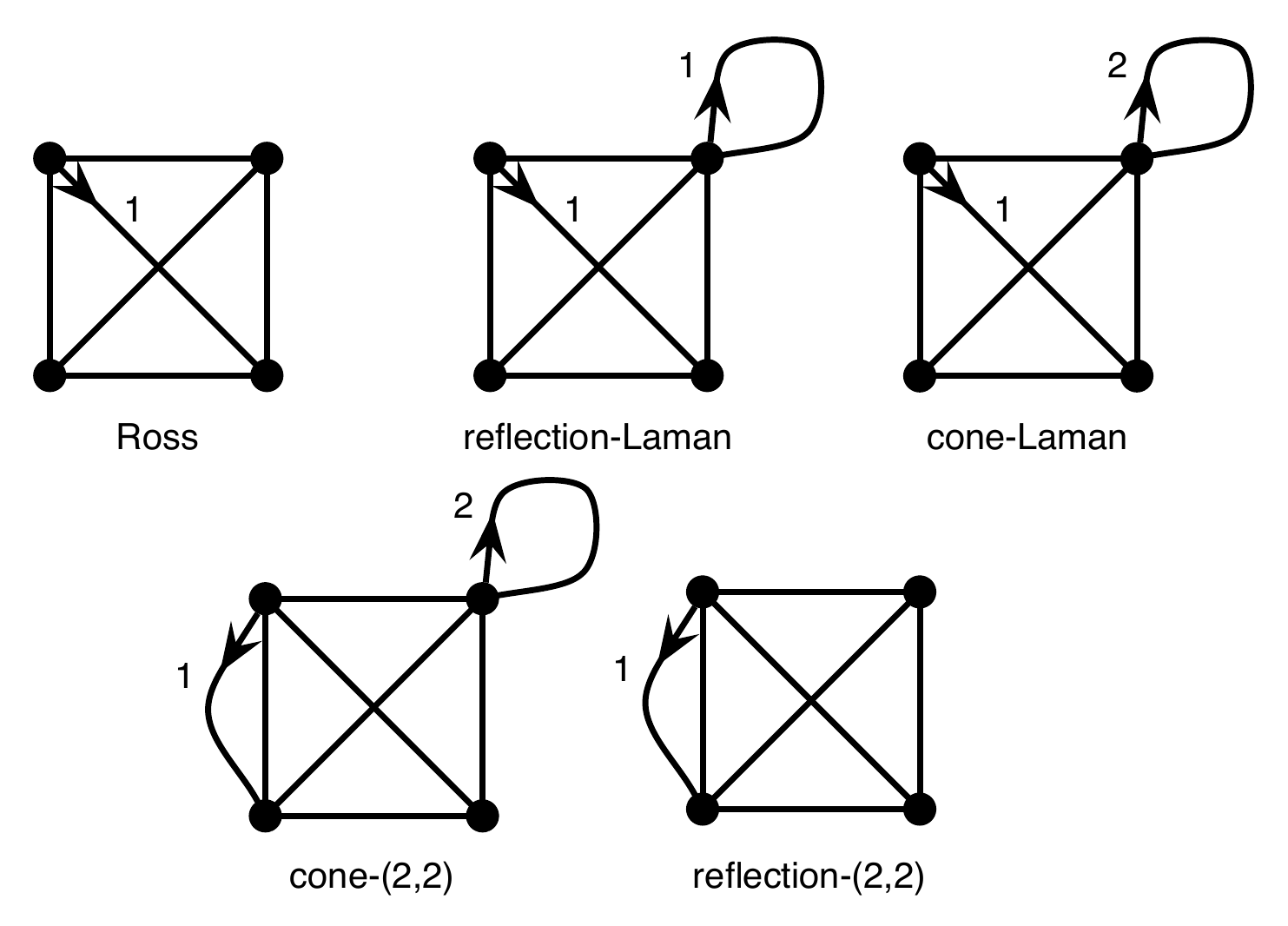}
\caption{\small Various examples of graphs. For the reflection-Laman and reflection-(2,2)
graphs, the colors lie in $\Z/2\Z$. For the rest, the colors lie in $\Z/k\Z$.   \normalsize }
\label{fig:examples}
\end{figure}

\section{Cone direction networks}\seclab{cone-direction-network}
In this section, we prove \theoref{cone-direction-network}.  The main step in the proof
is:
\begin{prop}\proplab{cone-collapse}
Let $(G,\bgamma)$ be a cone-$(2,2)$ graph.  Then every realization of a generic
direction network on $(G,\bgamma)$ is collapsed, with all vertices placed on the
rotation center.
\end{prop}
Next, we introduce \emph{colored direction networks}, which will be convenient to work with.

\subsection{The colored realization system}\seclab{colored-system}
The system of equations \eqref{dn-realization1}--\eqref{dn-realization2}
defining the realization space of a cone direction network
$(\tilde{G},\varphi,\vec d)$ is linear, and
as such has a well-defined dimension.  Let $(G,\bgamma)$ be the
colored quotient graph of $(\tilde{G},\varphi)$.

Since our setting is symmetric frameworks, we require that the assigned directions
also be symmetric. In other words, if $\vec d_{ij}$ is the direction for the edge $\tilde{i}_0 \tilde{j}_{\gamma_{ij}}$
in the fiber of $ij$, then $R_k^\gamma \vec d_{ij}$ is the direction for the edge $\tilde{i}_\gamma \tilde{j}_{\gamma+\gamma_{ij}}$
where here and throughout this section $R_k$ is the counter-clockwise rotation about the origin
through angle $2 \pi /k$.
Thus,  to specify a direction network on $(\tilde{G}, \varphi)$, we need only assign a direction to one edge in
each edge orbit.
Furthermore, since the directions and realizations must be symmetric, the system can be reduced to the following one where
the unknowns consist of $n$ points $\vec p_i$, one for each vertex $i$ of the quotient graph
where $\vec d_{ij} = \vec d_{\tilde{i}_0 \tilde{j}_{\gamma_{ij}}}$.
\begin{eqnarray}\eqlab{colored-system}
\iprod{R_k^{\gamma_{ij}}\cdot\vec p_j - \vec p_i}{\vec d_{ij}} = 0 & \qquad
\text{for all edges $ij\in E(G)$.}
\end{eqnarray}
\propref{cone-collapse} can be reinterpreted to say that the colored realization system \eqref{colored-system},
in matrix form, is a linear representation for the matroid of cone-$(2,2)$ graphs.  The representation
obtained via \propref{cone-collapse} is different than the one produced
by the matroid union construction for linearly representable matroids (see \cite[Proposition 7.16.4]{B86}).

\subsection{Genericity}\seclab{genericity}
Let $(G,\bgamma)$ be a colored graph with $m$ edges.  A statement about direction networks
$(\tilde{G},\varphi,\vec d)$ is \emph{generic} if it holds on the complement of a
proper algebraic subset of the possible direction assignments, which is canonically identified
with $\R^{2m}$. Some facts about generic statements that
we use frequently are:
\begin{itemize}
\item Almost all direction assignments are generic.
\item If a set of directions is generic, then so are all sufficiently
small perturbations of it.
\item If two properties are generic, then their intersection is as well.
\item The maximum rank of \eqref{colored-system} is a generic property.
\end{itemize}

This next proof is relatively standard.
\subsection{Proof that \propref{cone-collapse} implies \theoref{cone-direction-network}}\seclab{cone-direction-network-proof}
We prove each direction of the statement in turn.  Since it is technically
easier, we prove the equivalent statement on colored direction networks.

\paragraph{Cone-laman graphs generically have faithful realizations}
The proof in \cite[Section 15.3]{MT10} applies with small modifications.

\paragraph{Cone-laman circuits have collapsed edges}
For the other direction, we suppose that $(G,\bgamma)$ has $n$ vertices and is not cone-Laman.  If the number of edges $m$
is less than $2n-1$, then the realization space of any direction network is at least $2$-dimensional, so
it contains more than just rescalings.  Thus, we assume that $m\ge 2n - 1$.  In this case, $G$ is not
cone-Laman sparse, so it contains a cone-Laman circuit $(G',\bgamma)$.  Thus, we are reduced to showing that
any cone-Laman circuit has, generically, only realizations with collapsed edges, since these then force collapsed
edges in any realization of a  generic colored direction network on $(G,\bgamma)$.

There are two types of subgraphs $(G',\bgamma)$ that constitute minimal violations of cone-Laman sparsity:
\begin{itemize}
\item $(G',\bgamma)$ is a cone-$(2,2)$ graph.
\item $(G',\bgamma)$ has trivial $\rho$-image, and is a $(2,2)$-block.
\end{itemize}
If $(G',\bgamma)$ is a cone-$(2,2)$ graph, then \propref{cone-collapse} applies to it, and we are done.

For the other type, we may assume that the colors on $G'$ are zero by \lemref{uncolored-subgraph}
in which case the system \eqref{colored-system} on $(G', \bgamma)$ is equivalent to a direction
network on $G'$ as a finite, unsymmetric, uncolored graph. Thus, the Parallel Redrawing Theorem \cite[Theorem 4.1.4]{W96}
in the form \cite[Theorem 3]{ST10} applies directly to show that all realizations of $G'$ have only
collapsed edges.
\eop

The rest of this section proves \propref{cone-collapse}.

\subsection{Geometry of some generic linear projections}
We first establish some geometric results we need below.
Given a unit vector $\vec v\in \R^2$
and a point $\vec p \in \R^2$, we denote by $\ell(\vec v, \vec p)$ the affine line consisting of points
$\vec q$ where
\[
\vec q - \vec p = \lambda \vec v
\]
for some scalar $\lambda \in \R$. I.e. $\ell(\vec v, \vec p)$ is the line through $\vec p$ in direction $\vec v$.

\subsubsection{An important linear equation}
The following is a key lemma which will determine where certain points must lie when solving a cone direction network.
\begin{lemma}\lemlab{rotation}  Suppose $R$ is a non-trivial rotation about the origin, $\vec v^*$ is a unit
vector and $\vec p$ satisfies
$$(R-I) \vec p = \lambda \vec v^*$$
for some $\lambda \in \R$.  Then, for some $C \in \R$, we have $\vec p = C \vec v$
where $\vec v = R_{\pi/2} R^{-1/2} \vec v^*$, $R^{-1/2}$ is the inverse of a square root of $R$, and $R_{\pi/2}$ is the counter-clockwise rotation through angle $\pi/2$.
\end{lemma}
\begin{proof}
A computation shows that $(R - I) R^{-1/2} = R^{1/2}-R^{-1/2}$ is a scalar multiple of $R_{\pi/2}$. The
Lemma follows.
\end{proof}

\subsubsection{The projection $T(\vec v,\vec w,R)$}
Let $\vec v$ and $\vec w$ be unit vectors in $\R^2$, and $R$ some nontrivial rotation.  Denote
by $\vec v^*$ the vector $R^{1/2} \cdot\vec v^\perp$ for some choice of square root $R^{1/2}$ of $R$.

We define $T(\vec v,\vec w,R)$ to be the linear projection from $\ell(\vec v,0)$ to $\ell(\vec w,0)$
in the direction $\vec v^*$.  The following properties of $T(\vec v,\vec w,R)$ are straightforward.
\begin{lemma}\lemlab{Tvw-not-zero}
Let $\vec v$ and $\vec w$ be unit vectors, and $R$ a nontrivial rotation.  Then, the linear
map $T(\vec v,\vec w,R)$:
\begin{itemize}
\item Is defined if $\vec v^*$ is not in the same direction as $\vec w$.
\item Is identically zero if and only if $\vec v^*$ and $\vec v$ are collinear.
\end{itemize}
\end{lemma}

\subsubsection{The scale factor of $T(\vec v,\vec w,R)$}
The image $T(\vec v,\vec w,R)\cdot \vec v$ is equal to $\lambda\vec w$, for some scalar $\lambda$.
We define the \emph{scale factor} $\lambda(\vec v,\vec w,R)$ to be this $\lambda$. We need the following elementary fact
about the scaling factor of $T(\vec v,\vec w,R)$.

\begin{lemma} \lemlab{scalefactor-poly}
Let $\vec v$ and $\vec w$ be unit vectors such that $\vec v^*$ and $\vec w$ are linearly independent.  Then
the scaling factor $\lambda(\vec v,\vec w,R)$ of the linear map $T(\vec v,\vec w, R)$ is
given by
\[
\frac{ \langle \vec v, (\vec v*)^\perp \rangle}{\langle \vec w, (\vec v*)^\perp\rangle}.
\]
\end{lemma}
\begin{proof}
The map $T(\vec v, \vec w, R)$ is equivalent to the composition of:
\begin{itemize}
\item perpendicular projection from $\ell(\vec v, 0)$ to $\ell((\vec v*)^\perp, 0)$, followed by
\item the inverse of perpendicular projection $\ell(\vec w, 0) \to \ell((\vec v^*)^\perp, 0)$.
\end{itemize}
The first map scales the length of vectors by $\langle \vec v, (\vec v*)^\perp \rangle$ and the second by the reciprocal of $\langle \vec w, (\vec v*)^\perp\rangle$.
\end{proof}

From \lemref{scalefactor-poly} it is immediate that
\begin{lemma}\lemlab{scalefactor-blowup}
The scaling factor $\lambda(\vec v,\vec w, R)$ is identically $0$ precisely when $R$ is an order two rotation.
If $R$ is not an order $2$ rotation, then $\lambda(\vec v, \vec w, R)$ approaches infinity as $\vec v^*$
approaches $\pm \vec w$.
\end{lemma}

\subsubsection{Generic sequences of the map $T(\vec v,\vec w,R)$}
Let $\vec v_1,\vec v_2,\ldots, \vec v_n$ be unit vectors, and $S_1,S_2,\ldots, S_n$ be rotations of the form $R^i_k$ where
$R_k$ is a rotation of order $k$.
We define the linear
map $T(\vec v_1,S_1,\vec v_2,S_2,\ldots, \vec v_n,S_n)$ to be
\[
T(\vec v_1,S_1,\vec v_2,S_2,\ldots, \vec v_n,S_n) =
T(\vec v_n,\vec v_1,S_n) \circ T(\vec v_{n-1},\vec v_n,S_{n-1}) \circ\cdots \circ T(\vec v_1,\vec v_2,S_1)
\]
This next proposition plays a key role in the next section, where it is interpreted
as providing a genericity condition for cone direction networks.
\begin{prop}\proplab{cone-genericity}
Let $\vec v_1,\vec v_2,\ldots, \vec v_n$ be pairwise linearly independent unit vectors, and
$S_1,S_2,\ldots, S_n$ be rotations of the form $R^i_k$.
Then if the $\vec v_i$ avoid a proper algebraic subset
of $\left(\mathbb{S}^1\right)^n$ (that depends on the $S_i$),
the map $T(\vec v_1,S_1,\vec v_2,S_2,\ldots, \vec v_n,S_n)$ scales the length of
vectors by a factor of $\lambda\neq 1$.
\end{prop}
\begin{proof}
If any of the $T(\vec v_i,\vec v_{i+1},S_i)$ are identically zero, we are done, so we assume none of them
are.  The map $T(\vec v_1,S_1,\vec v_2,S_2,\ldots, \vec v_n,S_n)$ then scales vectors by a factor of:
\[
\lambda(\vec v_1,\vec v_2,S_1)\cdot\lambda(\vec v_2,\vec v_3,S_2)\cdots
\lambda(\vec v_{n-1},\vec v_n,S_{n-1})\cdot\lambda(\vec v_n,\vec v_1,S_n)
\]
which we denote $\lambda$.  That $\lambda$ is constantly one is a polynomial statement in the $\vec v_i$ by
\lemref{scalefactor-poly}, and so it is either always true or holds only on a proper algebraic subset of
$\left(\mathbb{S}^1\right)^n$.  This means it suffices to prove that there is one selection for the
$\vec v_i$ where $\lambda\neq 1$.

Without loss of generality, we may assume that
$T(\vec v_j,\vec v_{j+1},S_j)$ is defined for all $j$. Note that $|\lambda( \vec v_n, \vec v_1, S_n)|$ attains
an absolute (non-zero) minimum when $\vec v_n^* = \vec v_1^\perp$ by \lemref{scalefactor-poly} and
$|\lambda(\vec v_1, \vec v_2, S_1)|$ grows arbitrarily large as $\vec v_1^*$ moves towards $\vec v_2$.
Thus $|\lambda|$ grows arbitrarily large as $\vec v_1^*$ moves towards $\vec v_2$.
This means that $\lambda$ is not constantly $1$.
\end{proof}

\subsection{Proof of \propref{cone-collapse}}\seclab{cone-collapse-proof}
It will suffice to prove the existence of a single assignment of directions $\vec d$ to
a cone-$(2,2)$ graph $(G,\bgamma)$ for which we can show all the realizations $G(\vec p)$
of the resulting direction network are collapsed.  The generic statement is then immediate.
The strategy is to first decompose $(G,\bgamma)$ into cone-$(1,1)$ graphs using
\propref{cone-22-decomp}, assign directions to each connected cone-$(1,1)$ graph
which force a certain local geometric structure on $G(\vec p)$ (\secref{cone-collapse-directions}),
and then to show that the properties of the overlap graph imply the whole realization must
be collapsed (Sections \ref{sec:cone-collapse-2}--\ref{sec:cone-collapse-ge3}).

\subsubsection{Special direction networks on connected map-graphs}\seclab{cone-collapse-directions}
Let $(G,\bgamma)$ be a $\Z/k\Z$-colored graph that is a connected cone-$(1,1)$
graph.
We select and fix a base vertex $b\in V(G)$ that is on the (unique) cycle in $G$.
This next lemma provides the main ``gadget'' that we
use in the proof of \propref{cone-collapse} below.
\begin{lemma}\lemlab{cone11-directions}
Let $k\ge 2$, and let $(G,\bgamma)$ be a $\Z/k\Z$-colored graph that is a connected cone-$(1,1)$ graph with a base vertex $b$.
Let $\gamma\in \Z/k\Z$ be the $\rho$-image of the cycle in $G$, let
$\vec v$ be a unit vector, and let $\vec v^* = (R^{\gamma/2}_k\cdot \vec v)^\perp$.  We can assign directions $\vec d$ to the edges of
$G$ so that, in all realizations of the corresponding direction network $(\tilde{G}, \varphi, \vec d)$ on the lift:
\begin{itemize}
\item For each $\gamma' \in \Z/k\Z$, the point $\vec p_{\tilde{b}_{\gamma'}}$ lies
on the line $\ell(R^{\gamma'}_k \cdot \vec v, 0)$
\item The rest of the points all lie on the lines $\ell(R^{\gamma'} \vec v^*, \vec p_{\tilde{b}_{\gamma'}})$ as $\gamma'$ ranges over $\Z/k\Z$.
\end{itemize}
\end{lemma}
\begin{proof}
We assign
directions in the lift $\tilde{G}$ of $G$.  We start by selecting an edge $bi\in E(G)$ that is
incident on the base vertex $b$ and in the cycle in $G$. Then, $G-bi$ is a spanning tree $T$ of $G$.

Since $T$ is contractible, it lifts to $k$ disjoint copies of itself in $\tilde{G}$.  Let $\tilde{T}$
be the copy containing $\tilde{b}_0$.  Note that $\tilde{T}$ hits the fiber
over every edge in $G$ except for $bi$ exactly one time and the fiber
over every vertex exactly one time.

Recall that we only need to assign a direction to one edge in each orbit.
We first assign every edge in $\tilde{T}$ the direction $\vec v^*=(R^{\gamma/2}_k\cdot \vec v)^\perp$.
From this choice of directions already, it now follows by the connectivity of $T$
that in any realization of the cone direction network induced on the $\Gamma$-orbit
of $\tilde{T}$ any
point lies on some $\ell(R_k^{\gamma'} \vec v^*, \vec p_{\tilde{b}_{\gamma'}})$.

It only remains to specify a direction for some edge in the fiber of $bi$.
Select the edge in the fiber over $bi$ incident on the copy of $i$ in $\tilde{T}$.  Assign this
edge the direction $\vec v^*$ as well. Note that this edge is necessarily incident on the vertex
$\tilde{b}_{\gamma}$ since the cycle has $\rho$-image $\gamma$. The choice of direction
on this edge implies that
$$\vec p_{\tilde{b}_\gamma} - \vec p_{\tilde{b}_0} = R_k^\gamma \vec p_{\tilde{b}_0} - \vec p_{\tilde{b}_0} = \lambda v^*$$
for some scalar $\lambda$. It now follows from \lemref{rotation}, applied to the rotation $R_k^\gamma$,
that $\vec p_{\tilde{b}_0}$ lies on $\ell(\vec v, 0)$.
\end{proof}

\subsubsection{Proof of \propref{cone-collapse} for order $2$ rotations}\seclab{cone-collapse-2}
Decompose the cone-$(2,2)$ graph $(G,\bgamma)$ into two edge-disjoint (but not necessarily connected)
cone-$(1,1)$ graphs $X$ and $Y$, using \propref{cone-22-decomp}.
The order of the $\rho$-image of any cycle in either $X$ or $Y$ is always $2$, so the construction
of \lemref{cone11-directions} implies that, by assigning the same direction $\vec v$ to every
edge in $X$, every vertex in any realization lies on a single line through the origin in the direction of $\vec v$.
Similarly for edges in $Y$ assigned a direction $\vec w$ different than $\vec v$.

Since every vertex is at the intersection of two skew lines through the origin, the proposition is proved.
\eop

\subsubsection{Proof of \propref{cone-collapse} for rotations of order $k\ge 3$}\seclab{cone-collapse-ge3}
Using  \propref{cone-22-decomp}, fix a decomposition of the cone-$(2,2)$ graph $(G,\bgamma)$ into
two cone-$(1,1)$ graphs $X$ and $Y$ and choose base vertices for the connected components of $X$ and $Y$.
This determines an overlap graph $D$.  Since we deal with them independently of $X$ and $Y$, we define
$G_i$ to be the connected components of $X$ and $Y$, and recall that these partition the edges of $G$. We denote
the base vertex of $G_i$ by $b_i$. Since we use subscripts to denote the subgraph of the base vertex, in this section
we will use the notation $\gamma \cdot \tilde{b}_i$ as $\gamma$ ranges over $\Z/k\Z$ for the fiber
over vertex $b_i$.

\paragraph{Assigning directions}
For each $G_i$, select a unit vector $\vec v_i$ such that:
\begin{itemize}
\item For any $i, j$, we have $\vec v_i \neq R^\gamma_k \vec v_j$ for all $\gamma \in \Z/k\Z$.
\item For all choices $k_i$, the vectors $\vec w_i = R^{k_i}_k \vec v_i$ are generic in the sense of \propref{cone-genericity}.
\end{itemize}

Now, for each $G_i$ we assign directions as prescribed by \lemref{cone11-directions} where
$\vec v_i$ is the vector input into the lemma. This is
well-defined, since the $G_i$ partition the edges of $G$.  (They clearly overlap on the vertices--we will exploit
this fact below---but it does not prevent us from assigning edge directions independently.)

We define the resulting colored direction network to be $(G,\bgamma,\vec d)$ and the lifted
cone direction network $(\tilde{G},\varphi,\tilde{\vec d})$.  We
also define, as a convenience, the rotation $S_i$ to be $R_k^\gamma$ where $\gamma$ is the
$\rho$-image of the unique cycle in $G_i$.

\paragraph{Local structure of realizations}
Let $G_i$ and $G_j$ be distinct connected cone-$(1,1)$ components and suppose
that there is a directed edge $b_ib_j$ in the overlap graph $D$.  We have the following
relationship between $\vec p_{\tilde{b}_i}$ and $\vec p_{\tilde{b}_j}$ in realizations of
$(\tilde{G},\varphi,\tilde{\vec d})$.
\begin{lemma}\lemlab{cone-local-structure}
Let $\tilde{G}(\vec p)$ be a realization of the cone direction network $(\tilde{G},\varphi,\tilde{\vec d})$
defined above.  Let vertices $b_i$ and $b_j$ in $V(G)$ be the base vertices of $G_i$ and $G_j$, and
suppose that $b_ib_j$ is a directed edge in the overlap graph $D$.  Let $\gamma\cdot \tilde b_i$
be some vertex in the fiber of $b_i$. Then for some $\gamma'$, we have
$\vec p_{\gamma'\cdot \tilde b_j} = T(R^{\gamma'}_k \vec v_i, R^\gamma_k\vec v_j, S_i)\cdot \vec p_{\gamma \cdot\tilde b_i}$
\end{lemma}
The proof is illustrated in \figref{cone-local-structure-proof}.

\begin{figure}[htbp]
\labellist
\pinlabel $\vec v_1$ at 280 330
\pinlabel $\vec v_2$ at 215 230
\pinlabel $\vec p_{\tilde{b}_1}$ at 430 400
\pinlabel $\vec p_{\tilde{b}_2}$ at 220 130
\endlabellist
\centering
\includegraphics[width=.4\textwidth]{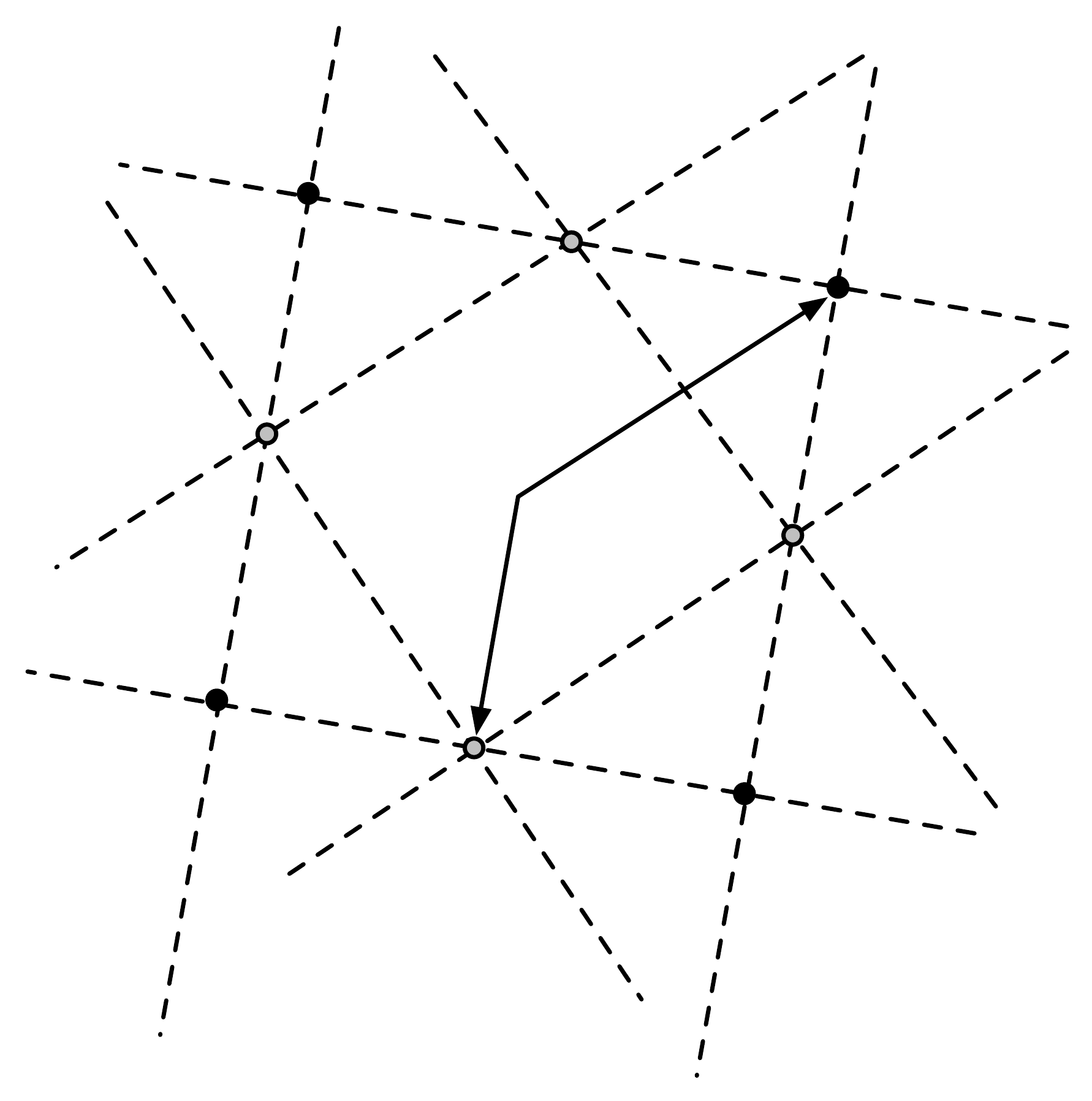}
\caption{\small
Example of the local structure of the proof of \propref{cone-collapse}; in this example, $R^2_4 \vec p_{\tilde{b}_2}$
is the image of $p_{\tilde{b}_1}$ via the projection $T(\vec v_1, R^2_4 \vec v_2, S_1)$. \normalsize
}
\label{fig:cone-local-structure-proof}
\end{figure}

\begin{proof}
By \lemref{cone11-directions}, the vertex $\vec p_{\gamma \cdot\tilde b_i}$ lies on the line $\ell(R^\gamma_k \vec v, 0)$, and
$\vec p_{\gamma'\cdot \tilde b_j} $ lies on $\ell(R^\gamma_k \vec v^*, \vec p_{\gamma \cdot\tilde b_i})$ for some $\gamma' \in \Z/k\Z$
since the vertex $b_j$ lies in the map-graph $G_i$. By \lemref{cone11-directions} applied to $G_j$, the vertex $\vec p_{\gamma' \cdot\tilde b_j}$
lies on the line $\ell(R^{\gamma'}_k \vec v, 0)$. This is exactly the situation captured by the map $T(R^{\gamma'}_k \vec v_i, R^\gamma_k\vec v_j, S_i)$.

\end{proof}

\paragraph{Base vertices on cycles in $D$ must be at the origin}
Let $b_i$ be the base vertex in $G_i$ that is also on a directed cycle in $D$.  The next step in the
proof is to show that all representatives in
$b_i$ must be mapped to the origin in any realization of $(\tilde{G},\varphi,\tilde{\vec d})$.
\begin{lemma}\lemlab{cone-cycle-vertices}
Let $\tilde{G}(\vec p)$ be a realization of the cone direction network $(\tilde{G},\varphi,\tilde{\vec d})$
defined above, and let $b_i\in V(G)$ be a base vertex that is also on a directed cycle in $D$ (one exists
by \propref{overlap-cycle}).  Then all vertices in the fiber over $b_i$ must be mapped to the origin.
\end{lemma}
\begin{proof}
Iterated application of \lemref{cone-local-structure} along the cycle in $D$ containing $b_i$ tells
us that any vertex in the fiber over $b_i$ is related to another vertex in the same fiber by a linear
map meeting the hypothesis of \propref{cone-genericity}.  This implies that if any vertex in the
fiber over $b_i$ is mapped to a point not the origin, some other vertex in the same fiber is mapped to a
point at a different distance to the origin.  This is a contradiction, since all realizations $\tilde{G}(\vec p)$
are symmetric with respect to $R_k$, so, in fact the fiber over $b_i$ was mapped to the origin.
\end{proof}

\paragraph{All base vertices must be at the origin}
So far we have shown that every base vertex $b_i$ that is on a directed cycle in the overlap graph $D$
is mapped to the origin in any realization $\tilde{G}(\vec p)$ of $(\tilde{G},\varphi,\tilde{\vec d})$.
However, since every base vertex is connected to the cycle in its connected component by a directed path in $D$,
we can show all the base vertices are at the origin.

\begin{lemma}\lemlab{cone-non-cycle-vertices}
Let $\tilde{G}(\vec p)$ be a realization of the cone direction network $(\tilde{G},\varphi,\tilde{\vec d})$
defined above.  Then all vertices in the fiber over $b_i$ must be mapped to the origin.
\end{lemma}
\begin{proof}
The statement is already proved for base vertices on a directed cycle in \lemref{cone-cycle-vertices}.
Any base vertex not on a directed cycle, say $b_i$, is at the end of a directed path which starts
at a vertex on the directed cycle.  Thus $\vec p_{\gamma \cdot b_i}$ is the image of $0$ under
some linear map, and hence is at the origin.
\end{proof}

\paragraph{All vertices must be at the origin}
The proof of \propref{cone-collapse} then follows from the observation that, if all the base vertices $b_i$
must be mapped to the origin in $\tilde{G}(\vec p)$, then \lemref{cone11-directions} implies that \emph{every}
vertex in the lift of $G_i$ lies on a family of $k$ lines intersecting at the origin.
Since every vertex is in the span of two of the $G_i$, and these families of lines
intersect only at the origin, we are done: $\tilde{G}(\vec p)$ must put all the points at the origin.
\eop

\section{Infinitesimal rigidity of cone frameworks} \seclab{cone-laman-proof}
The generic rigidity of a cone framework $(\tilde{G}, \varphi, \tilde{\ell})$
is a property of the underlying colored graph. To see this, note that we can identify the
realization space $\mathcal{R}(\tilde{G},\varphi,\tilde{\ell})$ with the solutions to the
following system, defined via the associated colored graph $(G, \bgamma)$:
\begin{eqnarray} \eqlab{cone-cont-rig}
||R^{\gamma_{ij}}_k \vec p_j - \vec p_i ||^2 = \ell_{ij}^2 & \text{for all edges $ij\in E(G)$}.
\end{eqnarray}
Here $\ell_{ij}$ is equal to the length of any lift of the edge $ij$ which, by symmetry,
is independent of the lift. We call the resulting object a \emph{colored framework}
$(G, \bgamma, \ell)$ on the quotient graph, and denote its realization space
by $\mathcal{R}(G, \bgamma, \ell)$.  Since the two spaces have the same dimension,
for any choice $V\subset V(\tilde{G})$ of vertex-orbit representatives in $\tilde{G}$,
the projection onto the set $\left(\vec p_i\right)_{i\in V}$ induces an algebraic isomorphism.

Let $G$ have $n$ vertices.  Computing the formal differential of \eqref{cone-cont-rig}, we see that a vector
$\vec v \in \R^{2n}$ is an infinitesimal motion if and only if
\begin{eqnarray} \eqlab{cone-inf-rig}
\langle R^{\gamma_{ij}}_k \vec v_j - \vec v_i, R^{\gamma_{ij}}_k \vec p_j - \vec p_i \rangle =
0 & \text{for all edges $ij\in E(G)$}.
\end{eqnarray}
We define $(G, \bgamma, \ell)$ to be {\it infinitesimally rigid} if the system \eqref{cone-inf-rig}
has rank $2n-1$. It is easy to see that infinitesimal rigidity of $(G, \bgamma, \ell)$ coincides
with that of the lift $(\tilde{G}, \varphi, \tilde{\ell})$.

\subsection{Framework genericity} \seclab{genericity2}
Our approach to genericity for cone frameworks is a small extension of the one for
finite frameworks from \cite{ST10}.  Let $(G,\bgamma,\ell)$ be a colored framework.  A
realization $G(\vec p)$ is generic if the rank of \eqref{cone-inf-rig} is maximum
among all the realizations.  Thus, the
non-generic subset of the colored realization space $\mathcal{R}(G,\bgamma,\ell)$ are
simply those realizations for which the (complexification of) the
system \eqref{cone-inf-rig} does not attain its maximal rank.  This is cut out by the minors
of the matrix form of \eqref{cone-inf-rig}, and so clearly algebraic.  Since the natural
homeomorphism $\mathcal{R}(\tilde{G}, \varphi, \tilde{\ell})\to \mathcal{R}(G, \bgamma, \ell)$,
is an algebraic map, the pre-image in $\mathcal{R}(\tilde{G}, \varphi, \tilde{\ell})$
of the non-generic subset of $\mathcal{R}(G, \bgamma, \ell)$ is also algebraic.

\subsection{Proof of \theoref{cone-laman}}
We prove necessity by inspecting \eqref{cone-inf-rig} and then sufficiency with
\theoref{cone-direction-network}.

\paragraph{The Maxwell direction}
Suppose that $(G, \bgamma)$ is a colored graph on $m = 2n-1$ edges
that is not cone-Laman-sparse. Thus $G$ contains (at least) one of two possible types of
cone-Laman-circuits which we call $G'$: a cone-$(2,2)$ graph or a
$(2,2)$-graph with trivial $\rho$-image. In the former case, the subgraph has $n$ vertices and $2n$ edges,
and since any framework has a trivial motion arising from rotation,
the system \eqref{cone-inf-rig} has a dependency.

In the latter case, by \lemref{uncolored-subgraph}, we may assume the edges are colored by $0$,
in which case the system \eqref{cone-inf-rig} is identical to the well-known rigidity matrix
for finite frameworks.  Since $G'$ is not $(2,3)$-sparse, the Maxwell-Laman Theorem provides
a dependency in \eqref{cone-inf-rig}. \eop

\paragraph{The Laman direction}
\theoref{cone-direction-network} implies that \eqref{colored-system} has, generically,
rank $2n-1$ if $(G,\bgamma)$ is cone-Laman.  This next proposition says that \eqref{cone-inf-rig}
has the same rank.
\begin{prop}\proplab{same-rank}
Let $G(\vec p)$ be a realization of a cone framework with colored graph $(G,\bgamma)$.
If $G(\vec p)$ is faithful and solves the direction network $(G, \bgamma , \vec d)$,
then the system \eqref{cone-inf-rig} for $G(\vec p)$ has the same rank as \eqref{colored-system} for
$(G, \bgamma, \vec d)$.
\end{prop}
\begin{proof}
Let $R_{\pi/2}$ be the counter-clockwise rotation through angle $\pi/2$. If $G(\vec p)$ faithfully
solves $(G, \bgamma, \vec d)$, then $R_k^{\gamma_{ij}} \vec p_j - \vec p_i = \alpha \vec d_{ij}$. A vector $\vec q$
is therefore an infinitesimal motion if and only if $\vec q^\perp$ is a solution to \eqref{colored-system} by this
computation:
\begin{eqnarray*}
\langle R^{\gamma_{ij}}_k \vec q_j - \vec q_i, \vec d_{ij} \rangle  =   0 & \iff \\
\langle R_{\pi/2} (R^{\gamma_{ij}}_k \vec q_j - \vec q_i), R_{\pi/2} \vec d_{ij} \rangle  =  0 & \iff \\
\langle  R^{\gamma_{ij}}_k R_{\pi/2} \vec q_j - R_{\pi/2} \vec q_i, R_{\pi/2} \vec d_{ij} \rangle  =  0 & \,
\end{eqnarray*}
Thus, \eqref{cone-inf-rig} and \eqref{colored-system} have isomorphic solution spaces
and hence the same rank.
\end{proof}

\paragraph{Remark}
That $R_{\pi/2}$ and $R^{\gamma_{ij}}_k$ commute is critical in the computation used to prove
\propref{same-rank}. The corresponding argument for reflection
frameworks would fail since $R^{\gamma_{ij}}_k$ would be replaced by a reflection which does not commute with
$R_{\pi/2}$. In fact, as we will see, the ranks of the two systems can be different in the reflection case.

\subsection{Rigidity for non-free actions} \seclab{rotnonfree}
We briefly remark here on symmetric frameworks with fixed vertices or inverted edges. If there
is an edge $\tilde{i} \tilde{j} \in \tilde{G}$ and a group element $\gamma \in \Gamma$ such that
$\gamma \cdot \tilde{i} \tilde{j} = \tilde{j} \tilde{i}$, then $\tilde{i}$ and $\tilde{j}$ descend to the
same vertex, say $i$, in the quotient. The inverted edge, in terms of rigidity, forces $\vec p_i$
and $R^\gamma \cdot \vec p_i$ to remain at a constant distance. This is the same constraint
as if there were a self-loop at $i$ with color $\gamma$ in the quotient graph $(G, \bgamma)$.
Thus, for every inverted edge, we simply put such a self-loop in the quotient graph and appeal
to \theoref{cone-laman}.

If there is a fixed vertex, say $\tilde{i}$, then in any realization it lies at the origin. (We may assume
without loss of generality there is only one fixed vertex.) Suppose $\tilde{j}$, and consequently
its $\Gamma$-orbit, are connected to $\tilde{i}$ by some edge (orbit). As a symmetric framework,
this forces $\vec p_{j}$ and its $\Gamma$-orbit to be the vertices of a regular $k$-gon with fixed distance
from the origin. This same constraint can be enforced by deleting the edge orbit to $\tilde{i}$ and
replacing it with the edges of the regular polygon. (See \figref{fixvert}.) Thus, we reduce the problem again to
the case of a free $\Gamma$-action.

\begin{figure}[htbp]
\centering
\includegraphics[width=.6\textwidth]{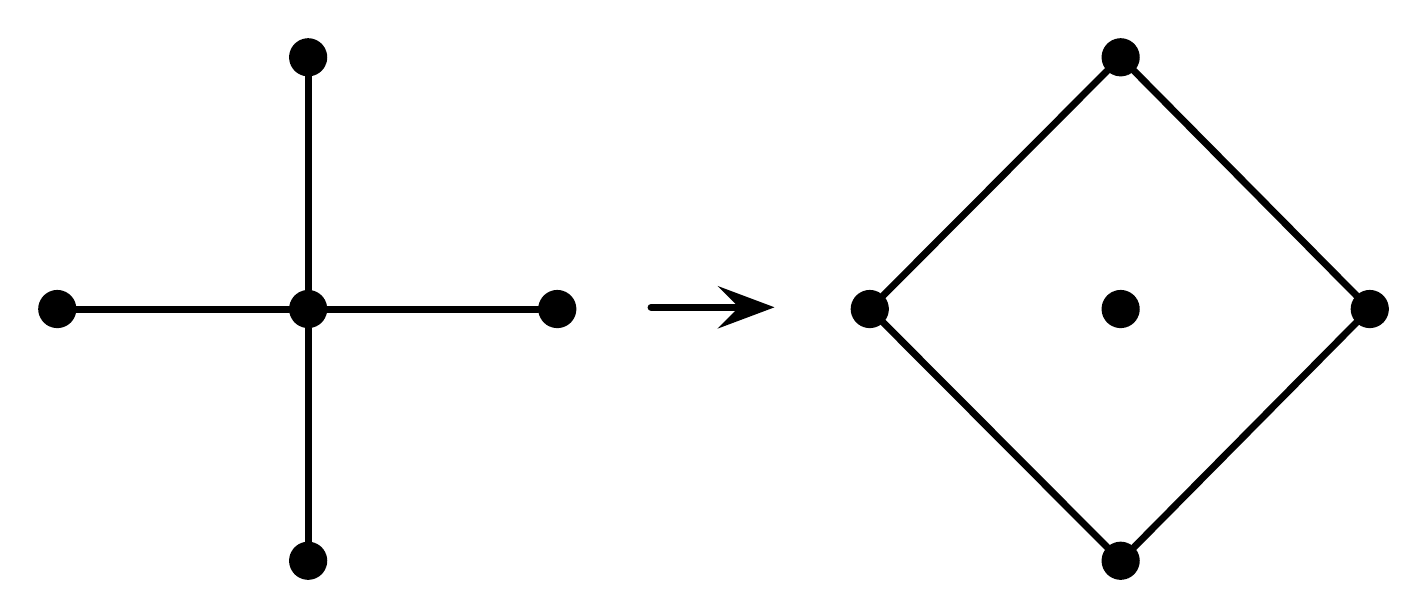}
\caption{\small Operation removing an edge orbit to a fixed vertex.   \normalsize }
\label{fig:fixvert}
\end{figure}

\section{Special pairs of reflection direction networks} \seclab{direction-network}
We recall, from the introduction, that for reflection direction networks, $\Z/2\Z$ acts on the plane
by some reflection through the origin. It is clear that we can reduce to the case where
the reflection is through the $y$-axis, and we make this assumption for the remainder of this section.

\subsection{Direction networks on Ross graphs}
We first characterize the colored graphs for which generic direction networks have
strongly faithful realizations.  A realization is \emph{strongly faithful} if no
two vertices lie on top of each other.  This is a stronger condition than simply being
faithful which only requires that edges not be collapsed.
\begin{prop}\proplab{ross-realizations}
A generic direction network  $(\tilde{G},\varphi,\vec d)$
has a unique, up to (vertical) translation and scaling, strongly faithful realization
if and only if its associated colored graph is a Ross graph.
\end{prop}
To prove \propref{ross-realizations} we expand upon the method from \cite[Sections 17--18]{MT10},
and use the following proposition.
\begin{prop}\proplab{reflection-22-collapse}
Let $(G,\bgamma)$ be a reflection-$(2,2)$ graph.  Then a generic direction network on the
symmetric lift $(\tilde{G},\varphi)$ of $(G,\bgamma)$ has only collapsed realizations.
\end{prop}
Since the proof of \propref{reflection-22-collapse} requires a detailed construction,
we first show how it implies \propref{ross-realizations}.
\subsection{Proof that \propref{reflection-22-collapse} implies \propref{ross-realizations}}
Let $(G,\bgamma)$ be a Ross graph, and assign directions $\vec d$ to the edges of $G$ such that,
for any extension $(G+ij,\bgamma)$ of $(G,\bgamma)$ to a reflection-$(2,2)$ graph as
in \propref{ross-adding}, $\vec d$ can be extended to a set of directions that is generic
in the sense of \propref{reflection-22-collapse}.  This is possible because there are a finite
number of such extensions.

For this choice of $\vec d$, the realization space of the direction network $(\tilde{G},\varphi,\vec d)$
is $2$-dimensional.  Since solutions to \eqref{colored-system} may be scaled or translated in the vertical
direction, all solutions to $(\tilde{G},\varphi,\vec d)$ are related by scaling and translation.  It
then follows that a pair of vertices in the fibers over $i$ and $j$ are either distinct from each other
in all non-zero solutions to \eqref{colored-system} or always coincide.  In the latter case, adding the
edge $ij$ with any direction does not change the dimension of the solution space, no matter
what direction we assign to it.  It then follows that the solution spaces of generic
direction networks on $(\tilde{G},\varphi,\vec d)$ and $(\widetilde{G+ij},\varphi,\vec d)$ have
the same dimension, which is a contradiction by \propref{reflection-22-collapse}.

For the opposite direction, suppose $(G, \bgamma)$ is not a Ross-graph. A proof similar to that in
\secref{cone-direction-network-proof} applies. If $m < 2n-2$, then dimension counting
tells us that the space of realizations cannot be unique up to translation and scaling.
If $m \geq 2n-2$, then $(G, \bgamma)$ has one of two types of circuits, either a
reflection-$(2,2)$-subgraph or a $(2,2)$-subgraph with trivial $\rho$-image.
In the former case, we are done by \propref{reflection-22-collapse}, and in the latter
case, the same proof from \secref{cone-direction-network-proof} applies.
\eop

\subsection{Proof of \propref{reflection-22-collapse}}
Let $(G,\bgamma)$ be a reflection-$(2,2)$ graph associated
to $(\tilde{G},\varphi)$.  It is sufficient to construct a set of
directions $\vec d$ such that the direction network $(\tilde{G},\varphi,\vec d)$
has only collapsed realizations.  In the rest of the proof, we
construct a set of directions $\vec d$ and then verify that the colored direction
network $(G,\bgamma,\vec d)$ has only collapsed solutions.  The proposition then
follows from the equivalence of colored direction networks with reflection
direction networks.

\paragraph{Combinatorial decomposition}
We apply \propref{reflection-22-nice-decomp} to decompose $(G,\bgamma)$ into a
spanning tree $T$ with all colors the identity and a reflection-$(1,1)$ graph $X$.  For
now, we further assume that $X$ is connected.

\paragraph{Assigning directions}
Let $\vec v$ be a direction vector that is not
horizontal or vertical.  For each edge $ij\in T$, set $\vec d_{ij} = \vec v$.
Assign all the edges of $X$ the vertical
direction.  Denote by $\vec d$ this assignment of directions.
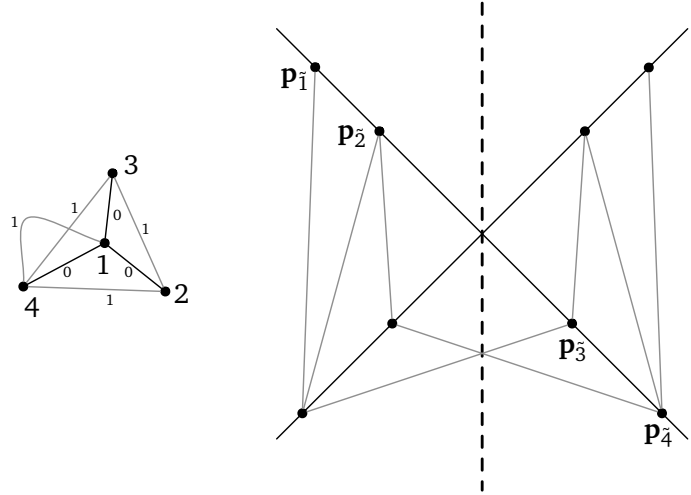
\begin{figure}[htbp]
\labellist
\pinlabel $1$ at 70 170
\pinlabel $2$ at 125 150
\pinlabel $3$ at 90 240
\pinlabel $4$ at 20 140
\pinlabel $\vec p_{\tilde{1}}$ at 205  300
\pinlabel $\vec p_{\tilde{2}}$ at 245 260
\pinlabel $\vec p_{\tilde{3}}$ at 400 110
\pinlabel $\vec p_{\tilde{4}}$ at 460 50
\tiny
\pinlabel $0$ at 80 205 %
\pinlabel $0$ at 45 165 %
\pinlabel $0$ at 88 165 %
\pinlabel $1$ at 50 210 %
\pinlabel $1$ at 100 195 %
\pinlabel $1$ at 75 145 %
\pinlabel $1$ at 8 200 %
\normalsize
\endlabellist
\centering
\includegraphics[width=.6\textwidth]{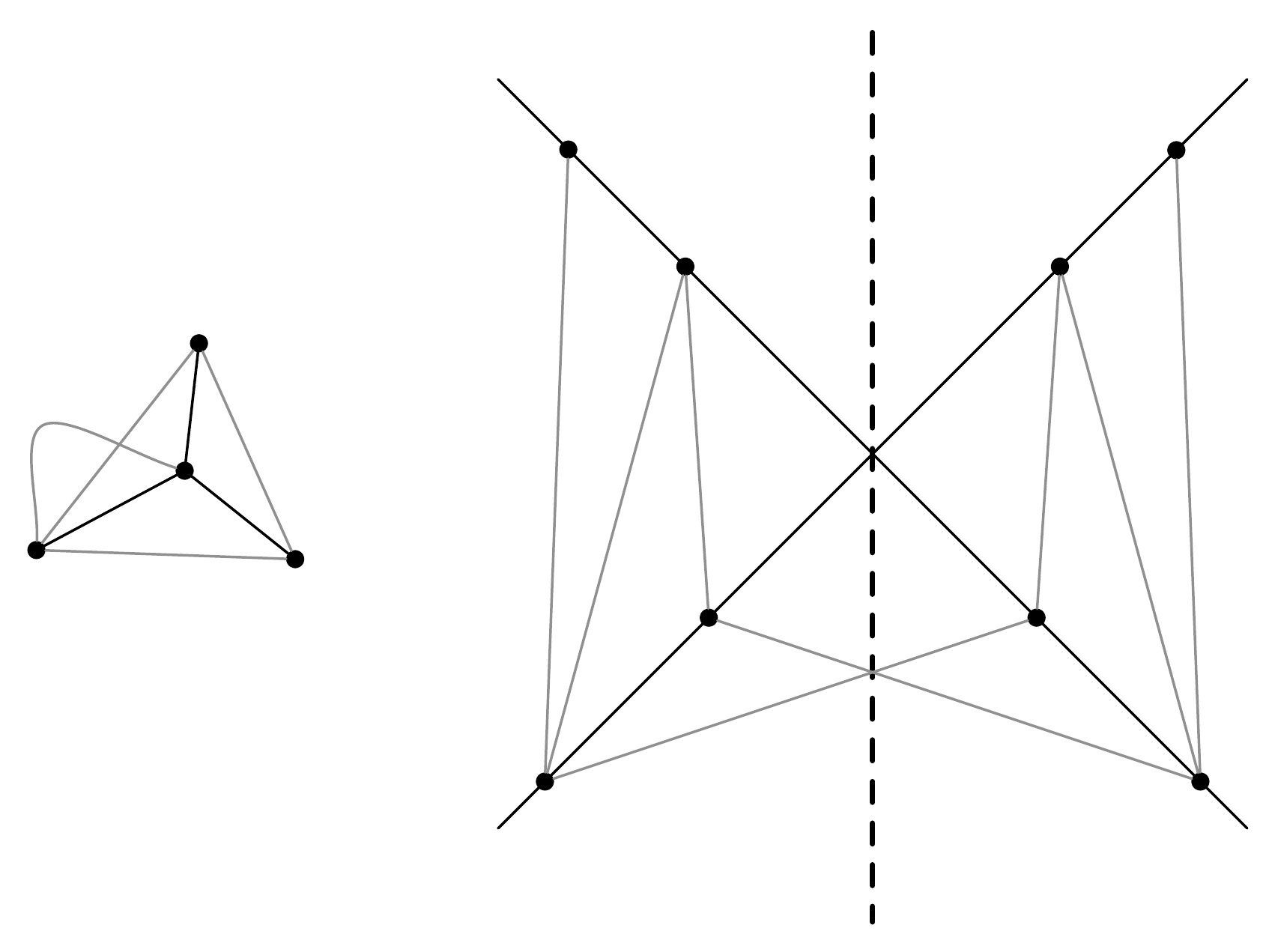}
\caption{\small Schematic for proof of \propref{reflection-22-collapse} and \lemref{RC-proof-1}: the $y$-axis
is shown as a dashed line. The corresponding colored graph is depicted on the left
with the tree indicated by black lines and the reflection-$(1,1)$ graph indicated by the gray lines.
The right-hand figure indicates a realization where only the directions on the tree
are enforced. If the gray lines are then forced to be vertical, the entire framework collapses
to the $y$-axis. If the gray lines are forced to be horizontal, the framework takes on the form in
\figref{ross-circuit-special-pair}. \normalsize }
\label{fig:ref-22-collapse}
\end{figure}
\paragraph{All realizations are collapsed}
We now show that the only realizations of $(\tilde{G},\varphi,\vec d)$ have
all vertices on top of each other.  By \propref{reflection-laman-decomp-lift},
$T$ lifts to two copies of itself, in $\tilde{G}$.  It then follows from
the connectivity of $T$ and the construction of $\vec d$ that, in any
realization, there is a line $L$ with direction $\vec v$ such that every vertex
of $\tilde{G}$ must lie on $L$ or its reflection.  Since the vertical direction
is preserved by reflection, the connectivity of the lift of $X$, again from
\propref{reflection-laman-decomp-lift}, implies that every vertex of $\tilde{G}$
lies on a single vertical line, which must be the $y$-axis by reflection symmetry.

Thus, in any realization of $(\tilde{G},\varphi,\vec d)$ all the vertices lie
at the intersection of $L$, the reflection of $L$
through the $y$-axis and the $y$-axis itself.  This is a single point, as
desired.  \figref{ref-22-collapse} shows a schematic of this argument.

\paragraph{$X$ does not need to be connected}
Finally, we can remove the assumption that $X$  was connected by repeating the argument
for each connected component of $X$ separately.
\eop

\subsection{Special pairs for Ross-circuits}
The full \theoref{reflection-direction-network} will reduce to the case of a Ross-circuit.
\begin{prop}\proplab{ross-circuit-pairs}
Let $(G,\bgamma)$ be a Ross circuit with lift $(\tilde{G},\varphi)$.  Then there
is an edge $i'j'$ with nonzero color such that, for a generic direction network $(\tilde{G'},\varphi,\vec d)$ with colored graph
$(G-i'j',\bgamma)$:
\begin{itemize}
\item For all faithful realizations of $(\tilde{G'},\varphi,\vec d)$, we have that
$\vec p_{\tilde{j}'_1} - \vec p_{\tilde{i}'_0}$ is a non-zero
vector with direction independent of the realization.
In particular, $(\tilde{G'},\varphi,\vec d)$ induces a well-defined direction on the edge
$i'j'$, which extends to an assignment of directions to the edges of $G$.
\item The direction networks $(\tilde{G},\varphi,\vec d)$
and $(\tilde{G},\varphi,(\vec d)^\perp)$ are a special pair.
\end{itemize}
\end{prop}
Before giving the proof, we describe the idea.  We are after sets of directions that
lead to faithful realizations of Ross-circuits.  By \propref{reflection-22-collapse},
these directions must be non-generic.  A natural way to obtain such a set of directions
is to discard an edge $ij$ from the colored quotient graph,
apply \propref{ross-realizations} to obtain a generic set of directions $\vec d'$ with a
strongly faithful realization $\tilde{G}'(\vec p)$, and then simply set the directions on the
edges in the fiber over $ij$ to be the difference vectors between the points.

\propref{ross-realizations} tells us that this procedure induces a well-defined
direction for the edge $ij$, allowing us to extend $\vec d$ from $G'$ to $G$ in a controlled way.
However, it does \emph{not} tell us that rank
of $(\tilde{G},\varphi,\vec d)$ will rise when the directions are turned
by angle $\pi/2$, and this seems hard to do directly.  Instead,
we construct a set of directions $\vec d$ so that $(\tilde{G},\varphi,\vec d)$ is rank deficient
and has realizations where $\vec p_i \neq \vec p_j$, and $(\tilde{G},\varphi,\vec d^\perp)$ is generic.
Then we make a perturbation argument to show the existence of a special pair.

The construction we use is, essentially, the one used in the proof of \propref{reflection-22-collapse}
but turned through angle $\pi/2$.  The key geometric insight is that horizontal edge directions
are preserved by the reflection, so the ``gadget'' of a line and its reflection crossing on the $y$-axis, as in
\figref{ref-22-collapse}, degenerates to just a single line.

\subsection{Proof of \propref{ross-circuit-pairs}}
Let $(G,\bgamma)$ be a Ross-circuit; recall that this implies that $(G,\bgamma)$ is a
reflection-Laman graph.

\paragraph{Combinatorial decomposition}
We decompose $(G,\bgamma)$ into a spanning tree $T$ and a
reflection-$(1,1)$ graph $X$ as in \propref{reflection-laman-decomp-lift}.  In particular,
we again have all edges in $T$ colored by the identity.  For now,
we \emph{assume that $X$ is connected}, and we fix $i'j'$ to be an edge that is on the cycle in $X$
with $\gamma_{i'j'}\neq 0$; such an edge must exist by the hypothesis that $X$ is
reflection-$(1,1)$. Let $G' = G\setminus i'j'$.  Furthermore, let $\tilde{T}_0$ and $\tilde{T}_1$
be the two connected components of the lift of $T$.  For a vertex $i \in G$, the lift
$\tilde{i}_j$ lies in $\tilde{T}_j$.  We similarly denote the
lifts of $i'$ and $j'$ by $\tilde{i}_0', \tilde{i}_1'$ and $\tilde{j}_0', \tilde{j}_1'$.

\paragraph{Assigning directions}
The assignment of directions is as follows: to the edges of $T$, we assign a direction
$\vec v$ that is neither vertical nor horizontal.  To the edges of $X$
we assign the horizontal direction.  Define the resulting direction network to be
$(\tilde{G},\varphi,\vec d)$, and the direction network induced on the lift of $G'$ to be
$(\tilde{G'},\varphi,\vec d)$.

\paragraph{The realization space of $(\tilde{G},\varphi,\vec d)$}
\figref{ref-22-collapse} and \figref{ross-circuit-special-pair} contain a schematic picture of the arguments that
follow.
\begin{lemma}\lemlab{RC-proof-1}
The realization space of $(\tilde{G},\varphi,\vec d)$ is $2$-dimensional
and parameterized by exactly one representative in the fiber over the
vertex $i$ selected above.
\end{lemma}
\begin{proof}
In a manner similar to
the proof of \propref{reflection-22-collapse}, the directions on the edges of $T$ force every
vertex to lie either on a line $L$ in the direction $\vec v$ or its reflection.  Since the lift
of $X$ is connected, we further conclude that all the vertices lie on a single horizontal line.
Thus, all the points $\vec p_{ \tilde{j}_0}$ are at the intersection of the same horizontal line and $L$
or its reflection. These determine the locations of the $\vec p_{ \tilde{j}_1}$, so the
realization space is parameterized by the location of $\vec p_{ \tilde{i}'_0}$.
\end{proof}
Inspecting the argument more closely, we find that:
\begin{lemma}
In any realization $\tilde{G}(\vec p)$ of $(\tilde{G},\varphi,\vec d)$,
all the $\vec p_{ \tilde{j}_0}$ are equal and all the $\vec p_{ \tilde{j}_1}$ are equal.
\end{lemma}
\begin{proof}
Because the colors on the edges of $T$ are all zero, it lifts to two copies of itself,
one of which spans the vertex set $\{\tilde{j_0} : j\in V(G)\}$ and one which spans
$\{\tilde{j_1} : j\in V(G)\}$.  It follows that in a realization, we have all the
$\vec p_{ \tilde{j}_0}$ on $L$ and the $\vec p_{ \tilde{j}_1}$ on the
reflection of $L$.
\end{proof}
In particular, because the color $\gamma_{i'j'}$ on the edge $i'j'$ is $1$, we obtain the following.
\begin{lemma}\lemlab{RC-proof-5}
The realization space of $(\tilde{G},\varphi,\vec d)$ contains points
where the fiber over the edge $i'j'$ is not collapsed.
\end{lemma}
\begin{figure}[htbp]
\centering
\includegraphics[width=.3\textwidth]{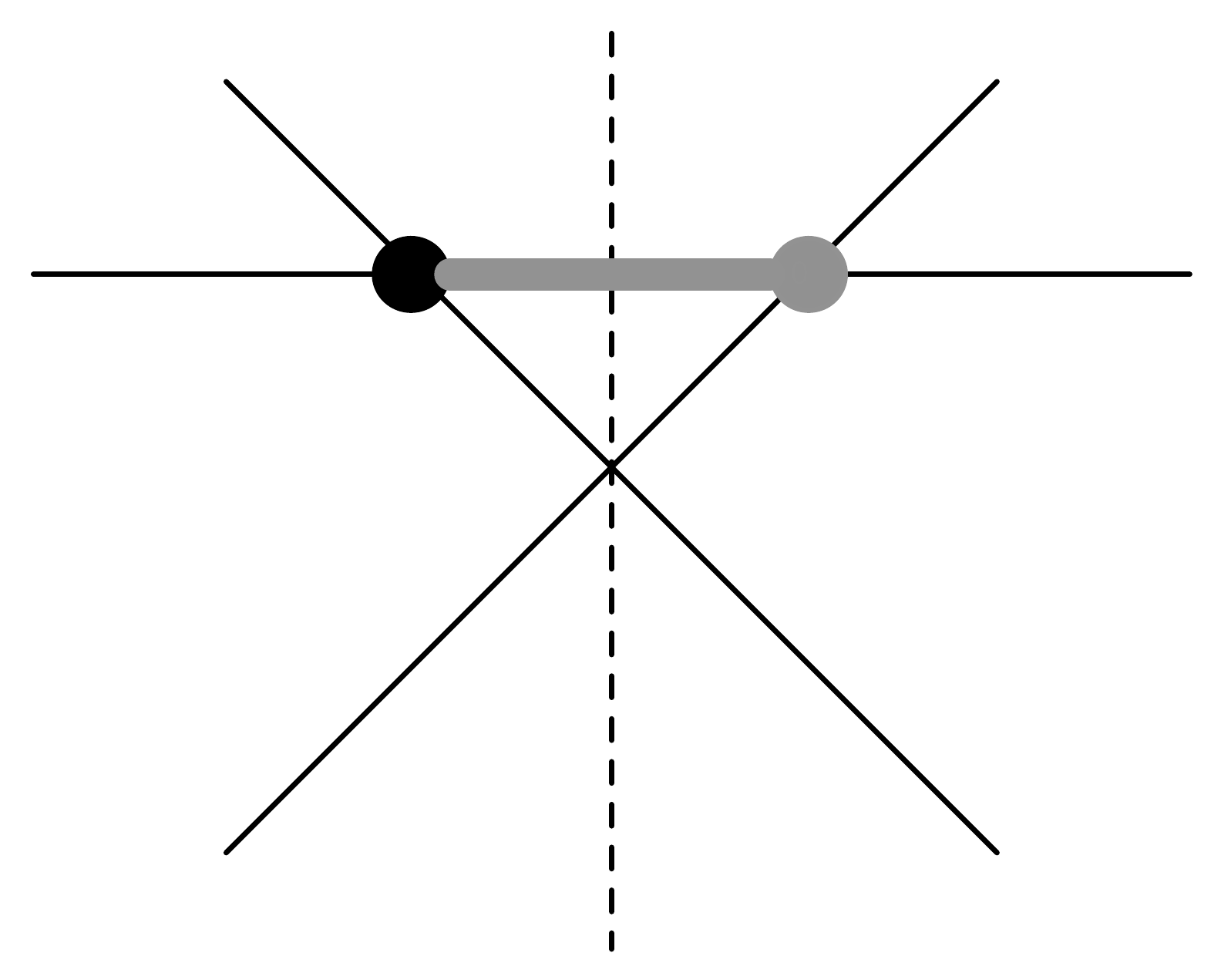}
\caption{\small Schematic of the proof of \propref{ross-circuit-pairs}: the $y$-axis
is shown as a dashed line.  The directions on the edges of the lift of the tree
$T$ force all the vertices to be on one of the two lines meeting at the $y$-axis.
The horizontal directions on the connected reflection-$(1,1)$ graph $X$ force the
point $\vec p_{ \tilde{j}_0}$ to be at the intersection marked by the black dot and
$\vec p_{ \tilde{j}_1}$ to be at the intersection marked by the gray one. The thick gray line
indicates a thick mass of horizontal edges. \normalsize}
\label{fig:ross-circuit-special-pair}
\end{figure}
\paragraph{The realization space of $(\tilde{G}',\varphi,\vec d)$}
The conclusion of \lemref{RC-proof-1} implies that the realization
system for $(\tilde{G},\varphi,\vec d)$ is rank deficient by one.
Next we show that removing the edge $i'j'$ results in a
direction network that has full rank on the colored graph $(G',\bgamma)$.
\begin{lemma}\lemlab{RC-proof-2}
The realization space of $(\tilde{G},\varphi,\vec d)$ is canonically identified
with that of $(\tilde{G}',\varphi,\vec d)$.
\end{lemma}
\begin{proof}
In the proof of \lemref{RC-proof-1}, it was not essential that $X$ lifts to a
connected subgraph of $\tilde{G}$. It was only required that $X$ spans the vertices,
and this is true of $X - i'j'$.
Since the two lifts of a vertex must always lie on the same horizontal line in a realization,
if any lift of $i$ and any lift of $j$ lie on the same horizontal line, then all lifts do. It is then
easy to conclude all points lie on a single horizontal line.
\end{proof}

\paragraph{The realization space of $(\tilde{G},\varphi,\vec d^\perp)$}
Next, we consider what happens when we turn all the directions by $\pi/2$.
\begin{lemma}\lemlab{RC-proof-3}
The realization space of $(\tilde{G},\varphi,\vec d^\perp)$ has only collapsed solutions.
\end{lemma}
\begin{proof}
This is exactly the construction used to prove \propref{reflection-22-collapse}.
\end{proof}

\paragraph{Perturbing $(\tilde{G},\varphi,\vec d)$}
To summarize what we have shown so far:
\begin{itemize}
\item[(a)] $(\tilde{G},\varphi,\vec d)$ has a $2$-dimensional realization space parameterized
by $\vec p_{\tilde{i}'_0}$ and identified with that of a full-rank direction network on  the Ross graph
$(G',\bgamma)$.
\item[(b)] There are points $\tilde{G}(\vec p)$ in this realization space where
$\vec p_{\tilde{i}'_0}\neq \vec p_{\tilde{j}'_1}$.
\item[(c)] $(\tilde{G},\varphi,\vec d^\perp)$ has a $1$-dimensional realization space containing only collapsed
solutions.
\end{itemize}
What we have not shown is that the realization space of $(\tilde{G},\varphi,\vec d)$
has \emph{faithful} realizations, since the ones we constructed all have many
coincident vertices.  \propref{ross-realizations} will imply the rest of the
theorem, provided that the above properties hold for any small perturbation of
$\vec d$, since some small perturbation of \emph{any} assignment of directions
to the edges of $(G',\bgamma)$ has only faithful realizations.
\begin{lemma}\lemlab{RC-proof-4}
Let $\vec{ \hat d'}$ be a perturbation of the directions $\vec d$ on the edges of $G'$ only.  If $\vec{ \hat d'}$
is sufficiently close to $\vec d|_{E(G')}$ ,
then there are realizations of the direction network
$(\tilde{G}',\varphi,\vec{ \hat d'})$ such that the direction of $\vec p_{ \tilde{j}'_1} - \vec p_{ \tilde{i}'_0}$ is nonzero and
a small perturbation of $\vec d_{ij}$.
\end{lemma}
\begin{proof}
The realization space is parameterized by $\vec p_{ \tilde{i}'_0}$ (for directions sufficiently close to $\vec d'$), and so $\vec p_{ \tilde{j}'_1}$ varies
continuously with the directions on the edges and $\vec p_{ \tilde{i}'_0}$.  Since there are realizations of $(\tilde{G}', \varphi, \vec d)$
with $\vec p_{ \tilde{i}_0} \neq \vec p_{ \tilde{j}_1}$, the Lemma follows.
\end{proof}
\lemref{RC-proof-4} implies that any sufficiently small perturbation of the directions assigned to the
edges of $G'$ gives a direction network that induces a well-defined direction on the edge $i'j'$ which
is itself a small perturbation of $\vec d_{i'j'}$.  Since the ranks of $(\tilde{G'},\varphi,\vec d')$
and  $(\tilde{G},\varphi,\vec d^\perp)$ are stable under small perturbations, this implies that
we can perturb $\vec d$ to a $\vec{\hat d}$ so that $\vec{\hat d}|_{E(G')}$ is generic in the sense of \propref{ross-realizations},
while preserving faithful realizability of $(\tilde{G},\varphi,\hat{\vec d})$ and full rank of the
realization system for $(\tilde{G},\varphi,\hat{\vec d}^\perp)$.  The Proposition is proved for when $X$ is
connected.

\paragraph{$X$ need not be connected}
The proof is then complete once we remove the additional assumption that $X$ was
connected.  Let $X$ have connected
components $X_1, X_2,\ldots,X_c$.  For each of the $X_i$, we can identify an edge
$(i'j')_k$ with the same properties as $i'j'$ above.

Assign directions to the tree $T$ as above.
For $X_1$, we assign directions exactly as above.  For each of the $X_k$ with $k\ge 2$,
we assign the edges of $X_k\setminus (i'j')_k$ the horizontal direction and $(i'j')_k$
a direction that is a small perturbation of horizontal.

With this assignment $\vec d$ we see that
for any realization of $(\tilde{G},\varphi,\vec d)$, each of the $X_k$, for $k\ge 2$
is realized as completely collapsed to a single point
at the intersection of the line $L$ and the $y$-axis.  Moreover,
in the direction network on $\vec d^\perp$, the directions on these $X_i$ are a small
perturbation of the ones used on $X$ in the proof of \propref{reflection-22-collapse}.
From this it follows that any realization of $(\tilde{G},\varphi,\vec d^\perp)$
is completely collapsed and hence full rank.

We now see that this new set of directions has properties (a), (b), and (c) above
required for the perturbation argument.  Since that argument makes no reference
to the decomposition, it applies verbatim to the case where $X$ is disconnected.
\eop

\subsection{Proof of \theoref{reflection-direction-network}}
The easier direction to check is necessity.
\paragraph{The Maxwell direction}
If $(G,\bgamma)$ is not reflection-Laman, then it contains either a
Laman-circuit with trivial $\rho$-image, or a violation of $(2,1)$-sparsity.
A violation of $(2,1)$-sparsity implies that the realization system \eqref{colored-system}
of $(\tilde{G},\varphi,\vec d^\perp)$ has a dependency, since the realization space is always at
least $1$-dimensional.

Suppose instead there is a Laman-circuit $G'$ with trivial $\rho$-image. Then any direction network
on $(G', \bgamma)$ is equivalent to a direction network on the (finite, uncolored, non-symmetric)
graph $G'$. (The lift $\tilde{G}'$ is two mirror images of $G'$.)
In this case, similar to \propref{same-rank}, $(G', \bgamma, \vec d)$ and
$(G', \bgamma, \vec d^\perp)$ have the same rank. Thus if $(G, \bgamma, \vec d^\perp)$ and
hence $(G', \bgamma, \vec d^\perp)$
has only collapsed realizations, so does $(G', \bgamma, \vec d)$ in which case $(G, \bgamma, \vec d)$
has no faithful realization.

\paragraph{The Laman direction}
Now let $(G,\bgamma)$ be a reflection-Laman graph and let $(G',\bgamma)$
be a Ross-basis of $(G,\bgamma)$.  For any edge $ij \notin G'$, adding it to
$G'$ induces a Ross-circuit\footnote{Recall that here we are using Ross-circuit
to refer to only one kind of circuit in the Ross matroid. The other type of circuit
cannot appear since reflection-Laman graphs do not have $(2,2)$ blocks
with trivial $\rho$-image.} which contains some edge $i'j'$ having the property
specified in \propref{ross-circuit-pairs}.  Note that $G' - ij +i'j'$ is again a Ross-basis.
We therefore can assume (after edge-swapping in this manner) for all $ij \notin G'$ that
$ij$ has the property from \propref{ross-circuit-pairs} in the Ross-circuit it induces.

We assign directions $\vec d'$ to the edges of $G'$ such that:
\begin{itemize}
\item The directions on each of the intersections of the Ross-circuits with $G'$ are generic in the sense
of \propref{ross-circuit-pairs}.
\item The directions on the edges of $G'$ that remain in the reduced graph $(G^*,\bgamma)$
are perpendicular to an assignment of directions on $G^*$ that is
generic in the sense of \propref{reflection-22-collapse}.
\item The directions on the edges of $G'$ are generic in the sense of \propref{ross-realizations}.
\end{itemize}
This is possible because the set of disallowed directions is the union of a finite number of
proper algebraic subsets in the space of direction assignments.  Extend to directions $\vec d$ on $G$
by assigning directions to the remaining edges as specified by \propref{ross-circuit-pairs}.  By construction,
we know that:
\begin{lemma}\lemlab{laman-1}
The direction network $(\tilde{G},\varphi,\vec d)$ has faithful realizations.
\end{lemma}
\begin{proof}
The realization space is identified with that of $(\tilde{G'},\varphi,\vec d')$, and $\vec d'$
is chosen so that \propref{ross-realizations} applies.
\end{proof}
\begin{lemma}\lemlab{laman-2}
In any realization of $(\tilde{G},\varphi,\vec d^{\perp})$, the Ross-circuits are realized with all their
vertices coincident and on the $y$-axis.
\end{lemma}
\begin{proof}
This follows from how we chose $\vec d$ and \propref{ross-circuit-pairs}.
\end{proof}
As a consequence of \lemref{laman-2}, and the fact that we picked $\vec d$ so that $\vec d^\perp$
extends to a generic assignment of directions $(\vec d^*)^\perp$ on the reduced graph $(G^*,\bgamma)$
we have:
\begin{lemma}
The realization space of $(\tilde{G},\varphi,\vec d^\perp)$ is identified with that of
$(\tilde{G^*},\varphi, (\vec d^*)^\perp)$ which, furthermore, contains only collapsed solutions.
\end{lemma}
Observe that a direction network for a single self-loop (colored $1$) with a generic direction only has solutions where
vertices are collapsed and on the $y$-axis.  Consequently, replacing a Ross-circuit with a single vertex
and a self-loop yields isomorphic realization spaces.  Since the reduced graph is reflection-$(2,2)$ by \propref{reduced-graph} and the directions
assigned to its edges were chosen generically for \propref{reflection-22-collapse},
that $(\tilde{G},\varphi,\vec d^\perp)$ has only collapsed solutions follows.
Thus, we have exhibited a special pair, completing the proof.
\eop

\paragraph{Remark} It can be seen that the realization space of a direction network as supplied by
\theoref{reflection-direction-network} has at least one degree of freedom for each edge that is not in a Ross basis.  Thus,
the statement cannot be improved to, e.g., a unique realization up to translation and scale.

\section{Infinitesimal rigidity of reflection frameworks} \seclab{reflection-laman-proof}
Let $(\tilde{G},\varphi,\bm{\ell})$ be a reflection framework, and let $(G,\bgamma)$
be the quotient graph with $n$ vertices.  The algebraic steps in this section are
similar to those in \secref{cone-laman-proof}.  For a reflection framework, the
realization space $\mathcal{R}(\tilde{G},\varphi,\tilde{\bm{\ell}})$,
defined by  \eqref{lengths-1}--\eqref{lengths-2} is canonically
identified with the solutions to:
\begin{eqnarray}\eqlab{colored-lengths}
||\gamma_{ij}\cdot \vec p_j - \vec p_i||^2 = \ell^2_{ij} & \qquad
\text{for all edges $ij\in E(G)$}.
\end{eqnarray}
As in \secref{direction-network}, we assume, without loss of generality, that $\Gamma$ acts
by reflections through the $y$-axis.

Computing the formal differential of \eqref{colored-lengths}, we obtain the system
\begin{eqnarray}\eqlab{colored-inf}
\iprod{\gamma_{ij}\cdot \vec p_j - \vec p_i}{\vec v_j - \vec v_i} = 0 & \qquad
\text{for all edges $ij\in E(G)$}
\end{eqnarray}
where the unknowns are the \emph{velocity vectors} $\vec v_i$.  A
realization is  \emph{infinitesimally rigid} if the system \eqref{colored-inf}
has rank $2n - 1$.  As in the case of cone frameworks, generically,
infinitesimal rigidity and rigidity coincide, and the non-generic set is
defined in the same way.

\subsection{Relation to direction networks}
Here is the core of the direction network method for reflection frameworks:
we can understand the rank of \eqref{colored-inf} in terms of a direction network.
\begin{prop}\proplab{rigidity-vs-directions}
Let $\tilde{G}(\vec p)$ be a realization of a reflection framework.
Define the direction $\vec d_{ij}$ to be
$\gamma_{ij}\cdot \vec p_j - \vec p_i$.  Then the rank of \eqref{colored-inf}
is equal to that of \eqref{colored-system} for the direction network
$(G,\bgamma,\vec d^{\perp})$.
\end{prop}
\begin{proof}
Exchange the roles of $\vec v_i$ and $\vec p_i$ in \eqref{colored-inf}.
\end{proof}

\subsection{Proof of \theoref{reflection-laman}} \seclab{refllamanproof}
The, more difficult, ``Laman direction'' of the Main Theorem follows immediately from
\theoref{reflection-direction-network} and \propref{rigidity-vs-directions}:
given a reflection-Laman graph \theoref{reflection-direction-network} produces a realization with
no coincident endpoints and a certificate that \eqref{colored-inf} has corank one. The
``Maxwell direction'' follows from a similar argument as that for \theoref{cone-laman}.
\eop

\paragraph{Remark}
The statement of \propref{rigidity-vs-directions} is \emph{exactly the same} as the analogous statement for
orientation-preserving cases of this theory.  What is different is that, for reflection frameworks,
the rank of $(G,\bgamma,\vec d^{\perp})$ is \emph{not}, the same
as that of $(G,\bgamma,\vec d)$.  By \propref{reflection-22-collapse}, the set of directions arising as
the difference vectors from point sets are \emph{always non-generic} on reflection-Laman graphs, so
we are forced to introduce the notion of a special pair.

\paragraph{Remark}
We can extend \theoref{reflection-laman} to $\Gamma$-actions with inverted edges
but which are otherwise free. Indeed, a similar argument as in \secref{rotnonfree}
applies here.

\bibliographystyle{plainnat}

\end{document}